\documentclass[11pt]{article}

\newcommand{\insieme}[1]{\left\{ #1 \right\}}
\usepackage{fullpage}

\usepackage{enumerate}
\usepackage{amsmath,amsfonts,amssymb,amsthm, bbold}
\usepackage{svg}
\usepackage{graphics,esint}
\usepackage{epsfig}
\usepackage{xcolor}
\usepackage{xspace}

\definecolor{pingreen}{rgb}{0,39,14}

\usepackage{xfrac}
\usepackage{comment}
\usepackage{hyperref}
\usepackage{cleveref}

\crefname{section}{§}{§§}
\Crefname{section}{§}{§§}

\def\vol{\mathrm{vol}}

\def\Acal{\mathcal{A}}

\def\rhs{r.h.s.\xspace}

\def\st{s.t.\xspace}

\DeclareMathOperator*{\argmin}{arg\,min}

\makeatletter
\newtheorem*{rep@theorem}{\rep@title}
\newcommand{\newreptheorem}[2]{%
\newenvironment{rep#1}[1]{%
 \def\rep@title{#2~\ref{##1}}%
 \begin{rep@theorem}}%
 {\end{rep@theorem}}}
\makeatother

\newtheorem{theorem}{Theorem}[section]

\newtheorem{lemma}[theorem]{Lemma}
\newtheorem{definition}[theorem]{Definition}

\newtheorem{remark}[theorem]{Remark}

\newtheorem{proposition}[theorem]{Proposition}

\newtheorem{corollary}[theorem]{Corollary}

\newreptheorem{theorem}{Theorem}

\newcommand{\R}{\mathbb{R}}
\newcommand{\Z}{\mathbb{Z}}

\newcommand{\Fcal}{\mathcal{F}}

\def\N{\mathbb N}
\def\eps{\varepsilon}

\def\per{\mathrm{Per}}

\def\eps{\varepsilon}

\def\d {\,\mathrm {d}}
\def\dx{\,\mathrm {d}x}
\def\dz{\,\mathrm {d}z}
\def\ds{\,\mathrm {d}s}
\def\du{\,\mathrm {d}u}
\def\dv{\,\mathrm {d}v}
\def\dt{\,\mathrm {d}t}

\def\FtL{\mathcal F_{\tau,L}}

\numberwithin{equation}{section}

\usepackage{authblk}

\author[1]{Sara Daneri\thanks{sara.daneri@gssi.it}}
\author[2]{Eris Runa\thanks{eris.runa@gmail.com}}
\affil[1]{Gran Sasso Science Institute, L'Aquila }
\affil[2]{Deutsche Bank, London}

  \title{Periodic striped configurations in the large volume limit}

\usepackage[
doi=false,
url=false,
eprint=false,
style=alphabetic,
maxcitenames=6,
maxbibnames=6,
maxalphanames=6,
]{biblatex}

\bibliography{allbib}

\date{}

\begin{document}

\maketitle

\begin{abstract}
	We show striped pattern formation in the large volume limit  for a class of generalized antiferromagnetic local/nonlocal interaction functionals in general dimension previously considered in \cite{gr,dr_arma,dr_vol} and in \cite{giuliani06_ising_model_with_long_range,gs_cmp} in the discrete setting. In such a model the relative strength between  the short range attractive term favouring pure phases and the long range repulsive term favouring oscillations is modulated by a parameter $\tau$. {For $\tau<0$ minimizers are trivial uniform states. It is conjectured that $\forall\,d\geq2$ there exists $0<\bar{\tau}\ll1$ such that for all $0<\tau\leq\bar{\tau}$ and for all $L>0$ minimizers are striped/lamellar patterns. In \cite{dr_arma} the authors prove the above for $L=2kh^*_\tau$, where $k\in\N$ and $h^*_\tau$ is the optimal period of stripes for a given $0<\tau\leq\bar{\tau}$. 
	The purpose of this paper is to show the validity of the conjecture for generic $L$. }
\end{abstract}

\section{Introduction}
\label{sec:introduction}

In this paper we consider the following class of functionals.

For  $d\geq1$, $L>0$, $\tau>0$, $p\geq d+2$, $\beta=p-d-1$, $E\subset\R^d$ $[0,L)^d$-periodic set, $Q_L=[0,L)^d$, define
\begin{equation}
   \label{eq:ftauel}
\mathcal F_{\tau,L}(E)=\frac{1}{L^d}\Big(-\per_1(E,Q_L)+\int_{\R^d} K_\tau(\zeta) \Big[\int_{\partial E \cap Q_L} \sum_{i=1}^d|\nu^E_i(x)| |\zeta_i|\d\mathcal H^{d-1}(x)-\int_{Q_L}|\chi_E(x)-\chi_E(x+\zeta)|\dx\Big]\d\zeta\Big),
\end{equation}
where  $\per_1(E,Q_L)=\int_{\partial E\cap Q_L}\|\nu^E(x)\|_1\d\mathcal H^{d-1}(x)$ is the $1$-perimeter of the set $E$ in the cube $Q_L$,  defined through the $1$-norm $\|z\|_1=\sum_{i=1}^d|z_i|$, and  $K_\tau(\zeta)=\tau^{-p/\beta}K_1(\zeta\tau^{-1/\beta})$, where $K_1(\zeta)=\frac{1}{(\|\zeta\|_1+1)^p}$.

The functional~\eqref{eq:ftauel} is obtained by suitably rescaling the local/nonlocal interaction functional
\begin{equation}
	\label{eq:ef}
	\bar\Fcal_{J,L}(E)=J\per_1(E,[0,L)^d)-\int_{\R^d}\int_{Q_L}|\chi_E(x)-\chi_E(x+\zeta)|K_1(\zeta)\dx\d\zeta,
\end{equation}
where $J=J_c-\tau$ and $J_c=\int_{\R^d}|\zeta_i|K_1(\zeta)\d\zeta$ is a critical constant such that for $J>J_c$ minimizers of $\bar\Fcal_{J,L}$ are trivial (i.e. $E=\emptyset$ or $E=\R^d$). 

While for $\tau<0$ minimizers are trivial, when $\tau=J_c-J$ is positive and small the competition between the short range attractive term of perimeter type and the long range repulsive term with power law interaction kernel in~\eqref{eq:ef} causes the breaking of symmetry w.r.t. coordinate permutations and  global minimizers are expected to be one-dimensional and periodic. Since the optimal period among periodic one-dimensional sets is of the order $\tau^{-1/\beta}$, and the optimal energy scales like $\tau^{(p-d)/\beta}$, in order to see striped patterns in a fixed box as $\tau\to0$ it is convenient to rescale the functional \eqref{eq:ef} in such a way that optimal stripes have width and energy of order $O(1)$, thus getting \eqref{eq:ftauel}.  

Showing symmetry breaking and pattern formation in more than one  space dimensions for local/nonlocal models retaining symmetry w.r.t. some rotational group turns out to be a challenging problem,  which up to now has been proved only in a few cases (\cite{gs_cmp,dr_arma,dr_siam,dkr,dr_therm,dr_vol}) for functionals retaining symmetry w.r.t. coordinate permutations and symmetric domains.

Defining the family of one-dimensional sets of period $L$ as 
\begin{equation}\label{eq:cl}
\mathcal C_L=\{E\subset\R^d:\,\text{ up to coordinate permuations } E=\hat E\times \R^{d-1} \text{ with }\hat E\subset\R\quad  \text{$L$-periodic}\}
\end{equation} and denoting the minimal energy attained by the functional in \eqref{eq:ftauel} on such sets as $L$ varies as follows 
\begin{equation}\label{eq:einfty}
	e_{\infty,\tau}=\inf_L\inf_{E\in \mathcal C_L}\Fcal_{\tau,L}(E),
\end{equation}
in~\cite{dr_arma} the authors proved that there exists $\hat{\tau}>0$ such that for every $0<\tau\leq\hat\tau$  the minimal energy $e_{\infty,\tau}$ is attained on periodic stripes with density $1/2$ and period $2h^*_\tau$ for a unique $h^*_\tau>0$ (the existence of a possibly non unique optimal period had been previously shown in \cite{chen2005periodicity, muller1993singular,ren2003energy,giuliani06_ising_model_with_long_range}). By periodic stripes with density $1/2$ and period $2h$ (simply called periodic unions of stripes of period $2h$ in~\cite{dr_arma}) we mean here sets which, up to permutations of coordinates and translations, are of the form
\begin{equation}\label{eq:stripes12}
	E=\bigcup_{k\in\Z}[2kh, (2k+1)h)\times\R^{d-1}.
\end{equation} 
Indeed, periodic stripes of density $1/2$ and period $2h^*_\tau$ not only minimize $\Fcal_{\tau, L}$ among one-dimensional sets in the class $\mathcal C_L$ when $L=2kh^*_\tau$,  but they turn out to be the unique global minimizers of $\Fcal_{\tau,L}$ among all $[0,L)^d$-periodic locally finite  perimeter sets in $\R^d$. 

More precisely, in \cite{dr_arma} the following result is proved

\begin{theorem}\label{T:1.3}
	Let $d\geq1$, $p\geq d+2$ and $h^{*}_{\tau}$ be the optimal stripes' width for fixed $\tau$. 
	Then there exists $\check \tau>0$, such that for every $0<\tau< \check \tau$, one has that for every $k\in \N$ and  $L = 2k h_{\tau}^{*}$,   the minimizers $E_{\tau}$ of $\FtL$ are optimal stripes of period $2h_{\tau}^{*}$ and density $1/2$. 
\end{theorem}

Hence, for every $L=2kh^*_\tau$ and $0<\tau\leq\check{\tau}$ with $\check{\tau}$ independent of $k$, minimizers of $\Fcal_{\tau, L}$ are, up to translations and permutation of coordinates, of the form~\eqref{eq:stripes12}. 

A  crucial point in the proof of Theorem \ref{T:1.3} was that the large volume limit structure of minimizers (namely for a fixed $0<\tau\ll1$  and then letting $L\to+\infty$) had to be performed on boxes whose sizes $L$ are even multiples of the optimal width $h^*_\tau$. On these boxes the energy of minimizers reaches the minimal value obtainable by periodic stripes. This allowed in \cite{dr_arma} to bound from below the energy of the set $E$ along one-dimensional slices in the different coordinate directions with the minimal energy density for periodic sets $e_{\infty,\tau}$. 

{An important question, especially for the applications,  that is left open in \cite{dr_arma}, is whether one can prove the above result for  boxes of arbitrary size $L$, not necessarily compatible with the optimal period. Indeed, in most applications the box size is predetermined by external factors. }

In this paper  we give a positive answer to the above question. More precisely, we prove the following

\begin{theorem}
	\label{thm:main}
	Let $d\geq1$, $p\geq d+2$. Then there exist $\bar{\tau}>0$ and $\bar L>0$ such that for all $0<\tau<\bar{\tau}$ and $L>2\bar L$  the $L$-periodic minimizers $E_{\tau}$ of $\FtL$ are optimal stripes of period $2h_{\tau,L}$ and density $1/2$, for some $h_{\tau,L}>0$.
\end{theorem}

{\begin{remark}
	The range $0<L\leq2\bar L$ (more in general, $L$ smaller than any fixed constant) is contemplated in \cite{dr_arma}[Theorem 1.2]. Therefore Theorem \ref{thm:main} above completes the proof of the conjecture for all $L>0$. 
\end{remark}

Among the many improvements, there are two main novelties in this paper:
\begin{itemize}
	\item We are able to devise one-dimensional optimization estimates depending on the minimal energy relative to the length of the intervals where such estimates are performed. Instead in \cite{dr_arma} global bounds involving the minimal energy density $e_{\infty,\tau}$ are used;
	\item As a consequence of length-dependent optimization estimates, on some intervals the energy could in principle be smaller the minimal energy density of the box $[0,L)^d$. We show that this is not the case, with a new argument that exploits the strict convexity of the energy density of optimal stripes w.r.t. their period. Such an argument enters in the proof of Theorem \ref{thm:main} in Section  \ref{Ss:main}.
\end{itemize}}



\subsection{Scientific context} 

Patterns emerge at nanoscale level in several physical/chemical systems. Surprisingly similar patterns among which droplets or stripes/lamellae can be found in different systems, with different types of interactions.  As pointed out in~\cite{sa}, the emergence of periodic regular structures is universally believed to stem from the competition between short range attractive and long range repulsive (SALR) interactions. Though observed in experiments and reproduced by simulations, a rigorous mathematical proof of pattern formation starting from symmetric functionals and domains in more than one space dimensions  is available only in a very few cases. The main difficulties lie in the symmetry breaking phenomenon and in the nonlocality of the interactions.

Below we report a (non-exhaustive) series of contributions on periodic stripes formation for symmetric functionals and domains in suitable regimes. In other regimes of competition between the short-range attractive and long-range repulsive forces with small volume constraints, strong indications of the emergence of patterns consisting of isolated droplets have been provided (see e.g. \cite{muratov2014isoperimetric,goldman2013gamma,goldman2014gamma,cicalese13_dropl_minim_isoper_probl_with}).

The one-dimensional setting is relatively well-understood.  Periodicity of global minimizers is known to hold for convex or reflection positive repulsive kernels (see  \cite{hubbard1978generalized, pokrovsky1978properties,kerimov1999uniqueness,muller1993singular,ren2003energy, chen2005periodicity,giuliani06_ising_model_with_long_range,giuliani08_period_minim_local_mean_field_theor,giuliani2009modulated}). 

In several space dimensions, the first characterization of ground states as periodic stripes was given in~\cite{gs_cmp} for a discrete version of the functional~\eqref{eq:ef} in the range of exponents $p>2d$. In the continuous setting, breaking of symmetry w.r.t. permutation of coordinates for the functional~\eqref{eq:ef} has been shown in~\cite{gr} in the range of exponents $p>2d$. The precise structure of minimizers of \eqref{eq:ef} (namely periodic stripes) in the wider range of exponents $p\geq d+2$ was given in~\cite{dr_arma}.    
In~\cite{ker} such a characterization of minimizers was proved to hold also in a small open range of exponents below $d+2$. We mention that in  physical applications the power law interactions have exponents smaller than or equal to $d+1$. This, together with the increased nonlocality of the problem, makes the problem of lowering the exponent of the kernel particularly interesting.  In \cite{dr_siam} it has been shown that also for repulsive kernels of screened-Coulomb (or Yukawa) type global minimizers are, in a suitable regime, periodic stripes. In~\cite{dkr} the authors consider the diffuse interface version of the functional~\eqref{eq:ef}  in a finite periodic box and prove one-dimensionality and periodicity of minimizers. In~\cite{dr_therm} the results in~\cite{dkr} are proved to hold in the large volume limit on boxes whose sizes are even multiples of an optimal period.  
In \cite{dr_vol} a characterization of minimizers for the functional \eqref{eq:ftauel} under the imposition of an arbitrary volume constraint $\alpha\in(0,1)$ was given. In the regime $0<\tau\ll1$ and when $L\gg1$ is an even multiple of the optimal period $h^*_{\tau,\alpha}$ of simple periodic stripes with density $\alpha$, minimizers among all $[0,L)^d$-periodic sets of density $\alpha$ happen to be given by sets of the form 
\begin{equation}\label{eq:stripesalpha}
	E=\bigcup_{k\in\Z}[2kh^*_{\tau,\alpha}, (2k+2\alpha)h^*_{\tau,\alpha})\times\R^{d-1}.
\end{equation} 
 Thus, even in the low density regime, stripes are the first type of pattern to emerge from the competition between attractive/repulsive forces in the range immediately  below the critical constant $J_c$.
 
 Regarding further fields of interest for pattern formation under attractive/repulsive forces in competition, we mention the following. Evolution problems of gradient flow type related to functionals with attractive/repulsive nonlocal terms in competition, both in presence and in absence of diffusion, are also well studied (see e.g.~\cite{carrillo2014derivation, carrillo2019blob, carrillo2011global, craig2020aggregation, craig2017nonconvex, daneri2020deterministic}). In particular, one would like to show stability of the gradient flows or of their deterministic particle  approximations around configurations which are periodic or close to periodic states.
Another interesting direction would be to extend our rigidity results to non-flat surfaces without interpenetration of matter as investigated for rod and plate theories  in~\cite{kupferman2014riemannian, lewicka2010shell, olbermann2017interpenetration}.

\section{Notation and preliminary results}

Let $d \geq 1$. On $\R^d$ let us denote by $\langle\cdot,\cdot\rangle$ the Euclidean scalar product and by $|\cdot|$ the Euclidean norm. Let $e_1, \dots, e_n$ be the canonical basis on $\R^d$. We will often employ slicing arguments, for this reason we need definitions concerning the $i$-th component. For $x \in \R^d$ let $x_i =\langle x,e_i\rangle $ and $x_i^{\perp} := x - x_ie_i$.  
Let 
$\|x\|_1=\sum_{i=1}^d|x_i|$ be the $1$-norm  and $\|x\|_\infty=\max_i|x_i|$ the $\infty$-norm. 
While writing slicing formulas,
with a slight abuse of notation we will sometimes identify 
$x_i\in[0,L)$  with the point $x_ie_i\in[0,L)^d$
and $\{x_i^{\perp}:\,x\in[0,L)^d\}$ with 
$[0,L)^{d-1}\subset\R^{d-1}$ 
so that  $x_i^{\perp} \in [0,L)^{d-1}$. 

Whenever $\Omega\subset\R^d$ is a measurable set, we denote by $\mathcal H^{d-1}(\Omega)$ its $(d-1)$-dimensional Hausdorff measure and by $|\Omega|$ its Lebesgue measure.

Given a measure $\mu$ on $\R^d$, we denote by $|\mu|$ its total variation. 

We recall that  a set $E\subset\R^d$ is of (locally) finite perimeter if the distributional derivative of its characteristic function $\chi_E$ is a (locally) finite measure. We denote by $\partial E$ the reduced boundary of $E$ and by $\nu^E$ its exterior normal. 

The anisotropic $1$-perimeter of $E$ is given by 
\[
\per_1(E,[0,L)^d):=\int_{\partial E\cap [0,L)^d}\|\nu^E(x)\|_1\d\mathcal H^{d-1}(x)
\]
and, for $i\in\{1,\dots,d\}$ 
\begin{equation}
	\label{eq:perI}
	\per_{1i}(E,[0,L)^d)=\int_{\partial E\cap [0,L)^d}|\nu^E_i(x)|\d\mathcal H^{d-1}(x),
\end{equation}
thus $\per_1(E,[0,L)^d)=\sum_{i=1}^d\per_{1i}(E,[0,L)^d)$. 

For $i\in\{1,\dots,d\}$,  we define the one-dimensional slices of $E\subset\R^d$ in direction $e_i$ by

\[
E_{x_i^\perp}:=\bigl\{s\in[0,L):\,se_i+ x_i^\perp\in E\bigr\}.
\]

Whenever $E$ is a set of locally finite perimeter, for a.e. $x_i^\perp$ its slice $E_{x_i^\perp}$ is a set of locally finite perimeter in $\R$ and the following slicing formula  holds for every $i\in\{1,\dots,d\}$
\[
\per_{1i}(E,[0,L)^d)=\int_{\partial E\cap [0,L)^d}|\nu^E_i(x)|\d\mathcal H^{d-1}(x)=\int_{[0,L)^{d-1}}\per_1(E_{x_i^\perp},[0,L))\dx_i^\perp.
\]

Consider  $E\subset\R$  a set of locally finite perimeter and $s\in\partial E$ a point in the relative boundary of $E$. We will denote by 
\begin{equation}
	\label{eq:s+s-}
	\begin{split}
		s^+ &:= \inf\{ t' \in \partial E, \text{with } t' > s  \} \\ s^- &:= \sup\{ t' \in \partial E, \text{with } t' < s  \}. 
	\end{split}
\end{equation}

We will also apply slicing on small cubes, depending on $l$, around a point. Therefore we introduce the following notation.
For $r> 0$ and $x^{\perp}_i$ we let $Q_{r}^{\perp}(x^\perp_{i}) = \{z^\perp_{i}:\, \|x^{\perp}_{i} - z^{\perp}_{i} \|_\infty \leq r  \}$ or we think of $x_i^\perp\in[0,L)^{d-1}$ and $Q_r^\perp(x_i^\perp)$ as a subset of $\R^{d-1}$. 
We denote also by $Q^i_r(t_i)\subset\R$ the interval of length $r$ centred in $t_i$.

From~\cite{gr,dr_arma} we recall that, using the equality $|\chi_E(x)-\chi_E(x+\zeta)|=|\chi_E(x)-\chi_E(x+\zeta_ie_i)|+|\chi_E(x+\zeta_ie_i)-\chi_E(x+\zeta)|-2|\chi_E(x)-\chi_E(x+\zeta_ie_i)||\chi_E(x+\zeta_ie_i)-\chi_E(x+\zeta)|$ and $Q_L$-periodicity,  the following lower bound holds.

\begin{align}
	\label{eq:gstr1}
	\Fcal_{\tau,L} (E) 
	&\geq-\frac{1}{L^d}\sum_{i=1}^{d}\per_{1i}(E,[0,L)^d) + \frac{1}{L^d}\sum_{i=1}^{d} \Big[\int_{[0,L)^d\cap \partial E} \int_{\R^d} |\nu^{E}_{i} (x)| |\zeta_{i} | K_{\tau}(\zeta) \d\zeta \d\mathcal H^{d-1}(x) \notag\\ 
	&- \int_{[0,L)^d } \int_{\R^d} |\chi_{E}(x + \zeta_ie_i)  - \chi_{E}(x) | K_{\tau}(\zeta) \d\zeta\dx \Big] \notag\\ 
	&+ \frac{2}{d} \frac{1}{L^d}\sum_{i=1}^d \int_{[0,L)^d } \int_{\R^d} |\chi_{E}(x + \zeta_{i}e_i) -\chi_{E}(x) | | \chi_{E}(x + \zeta^{\perp}_i) - \chi_{E}(x) | K_{\tau}(\zeta)
	\d\zeta \dx.
\end{align}

Notice that in~\eqref{eq:gstr1} equality holds whenever the set $E$ is a union of stripes.  Thus, proving that unions of stripes with density $1/2$ are the minimizers of the r.h.s. of~\eqref{eq:gstr1} implies that they are the minimizers for $\Fcal_{\tau,L}$.

Let us define
\[
\widehat K_{\tau}(\zeta_i)=\int_{ \R^{d-1}}K_{\tau}(\zeta_ie_i+\zeta_i^\perp)\d\zeta_i^\perp.
\]

As in Section 7 of~\cite{dr_arma} we further decompose the r.h.s. of~\eqref{eq:gstr1} as follows.

\begin{equation}\label{eq:decomp_r}
	\begin{split}
		&-\frac{1}{L^d}\per_{1i}(E,[0,L)^d)+ \frac{1}{L^d}\Big[\int_{[0,L)^d\cap \partial E} \int_{\R^d} |\nu^{E}_{i} (x)| |\zeta_{i} | K_{\tau}(\zeta)\d\zeta\d\mathcal H^{d-1}(x) \\& - \int_{[0,L)^d } \int_{\R^d} |\chi_{E}(x + \zeta_ie_i)  - \chi_{E}(x) | K_{\tau}(\zeta) \d\zeta \dx \Big]
		= \frac{1}{L^d}\int_{[0,L)^{d-1}}  \sum_{s\in  \partial E_{t^{\perp}_{i}}\cap [0,L]} r_{i,\tau}(E,t_{i}^{\perp},s)
		\dt^{\perp}_i, 
	\end{split}
\end{equation}
where  for $s\in \partial E_{t_{i}^{\perp}}$ 
\begin{equation}\label{eq:ritau}
	\begin{split}
		r_{i,\tau}(E, t_{i}^{\perp},s) := -1 + \int_{\R} |\zeta_{i}| \widehat K_\tau (\zeta_{i})\d\zeta_i &- \int_{s^-}^{s}\int_{0}^{+\infty} |\chi_{E_{t_{i}^{\perp}}}(u + \rho) - \chi_{E_{t_{i}^{\perp}}}(u) | \widehat K_\tau (\rho)\d\rho \du  \\ & - \int_{s}^{s^+}\int_{-\infty}^{0} |\chi_{E_{t_{i}^{\perp}}}(u + \rho) - \chi_{E_{t_{i}^{\perp}}}(u) |\widehat K_\tau (\rho) \d\rho \du.\\ 
	\end{split}
\end{equation}
and $s^-<s<s^+$ are as in~\eqref{eq:s+s-}.

Defining
\begin{equation} 
	\label{eq:defFE}
	\begin{split}
		f_{E}(t_{i}^{\perp},t_{i},\zeta_{i}^{\perp},\zeta_{i}): =  |\chi_{E}(t_ie_i+t_i^\perp + \zeta_{i}e_i )  - \chi_{E}(t_ie_i+t_i^\perp)| |\chi_{E}(t_ie_i+t_i^\perp + \zeta^{\perp}_{i}) - \chi_{E}(t_ie_i+t_i^\perp) |,
	\end{split}
\end{equation} 
one has that
\begin{align}
	\label{eq:decomp_double_prod}
	\frac{2}{d} \frac{1}{L^d}\sum_{i=1}^d &\int_{[0,L)^d } \int_{\R^d} |\chi_{E}(x + \zeta_{i}e_i) -\chi_{E}(x) | | \chi_{E}(x + \zeta^{\perp}_i) - \chi_{E}(x) | K_{\tau}(\zeta)
	\d\zeta \dx=\notag\\
	&=\frac{2}{d}\frac{1}{L^d} \int_{[0,L)^d } \int_{\R^d} f_{E}(t_{i}^{\perp},t_{i},\zeta_{i}^{\perp},\zeta_{i}) K_{\tau}(\zeta)  \d\zeta \dt\notag\\
	& =\frac{1}{L^d}  \int_{[0,L)^{d-1}} \sum_{s\in \partial E_{t_{i}^{\perp}}\cap [0,L]} v_{i,\tau}(E,t_{i}^{\perp},s)\dt_{i}^\perp + \frac{1}{L^d}\int_{[0,L)^d} w_{i,\tau}(E,t_i^\perp,t_i) \dt,
\end{align}
where
\begin{equation}\label{eq:witau}
	{w}_{i,\tau}(E,t_{i}^{\perp},t_{i}) = \frac{1}{d}\int_{\R^d}  
	f_{E}(t_{i}^{\perp},t_{i},\zeta_{i}^{\perp},\zeta_{i}) K_\tau(\zeta)  \d\zeta. 
\end{equation}
and
\begin{equation}\label{eq:vitau}
	v_{i,\tau}(E,t_{i}^{\perp},s) =  \frac{1}{2d}\int_{s^{-}}^{s^{+}} \int_{\R^{d}} f_{E}(t^{\perp}_{i},u,\zeta^{\perp}_{i},\zeta_{i}) K_\tau(\zeta) \d\zeta\du.
\end{equation}

Hence, putting together~\eqref{eq:decomp_r} and~\eqref{eq:decomp_double_prod} one has the following decomposition

\begin{align}
	\label{eq:decomposition}
	\Fcal_{\tau,L} (E) 
	&\geq \frac{1}{L^d}\int_{[0,L)^{d-1}}  \sum_{s\in  \partial E_{t^{\perp}_{i}}\cap [0,L]} r_{i,\tau}(E,t_{i}^{\perp},s)
	\dt^{\perp}_i\notag\\
	&+\frac{1}{L^d}  \int_{[0,L)^{d-1}} \sum_{s\in \partial E_{t_{i}^{\perp}}\cap [0,L]} v_{i,\tau}(E,t_{i}^{\perp},s)\dt_{i}^\perp \notag\\
	&+ \frac{1}{L^d}\int_{[0,L)^d} w_{i,\tau}(E,t_i^\perp,t_i) \dt.
\end{align}

The term $r_{i,\tau}$ penalizes oscillations with high frequency  in direction $e_i$, namely sets $E$ whose slices in direction $e_i$ have boundary points at small minimal distance (see Lemma~\ref{rmk:stimax1}). The term $v_{i,\tau}$  penalizes oscillations in direction $e_i$ whenever the neighbourhood of the point in $\partial E\cap{Q_l(z)}$ is close in $L^1$ to a stripe oriented along $e_j$ (see Proposition~\ref{lemma:stimaContributoVariazionePiccola}).

For every cube  $Q_{l}(z)$, with $l<L$ and $z\in[0,L)^d$, define now the following localization of $\Fcal_{\tau,L}$
\begin{equation} 
	\label{eq:fbartau}
	\begin{split}
		\bar{F}_{i,\tau}(E,Q_{l}(z)) &:= \frac{1}{l^d  }\Big[\int_{Q^{\perp}_{l}(z_{i}^{\perp})} \sum_{\substack{s \in \partial E_{t_{i}^{\perp}}\\ t_{i}^{\perp}+se_i\in Q_{l}(z)}} (v_{i,\tau}(E,t_{i}^{\perp},s)+ r_{i,\tau}(E,t_{i}^{\perp},s)) \dt_{i}^{\perp} + \int_{Q_{l}(z)} {w_{i,\tau}(E,t_{i}^{\perp}, t_i) \dt}\Big],\\
		\bar{F}_{\tau}(E,Q_{l}(z)) &:= \sum_{i=1}^d\bar F_{i,\tau}(E,Q_{l}(z)).
	\end{split}
\end{equation} 

The following inequality holds:
\begin{equation}
	\label{eq:gstr14}
	\begin{split}
		\Fcal_{\tau,L}(E) \geq \frac{1}{L^d} \int_{[0,L)^d}  
		\bar{F}_{\tau}(E,Q_{l}(z)) \dz. 
	\end{split}
\end{equation}
Since in~\eqref{eq:gstr14} equality holds for unions of stripes, in order to prove Theorem~\ref{thm:main} one can reduce to show that the minimizers of its right hand side  are periodic optimal stripes  of density $1/2$ provided $\tau$ and $L$ satisfy the conditions of the theorem.

In the next definition we define a quantity which measures the $L^1$ distance of a set from being a union of stripes.

\begin{definition}
	\label{def:defDEta}
	For every $\eta$ we denote by $\Acal^{i}_{\eta}$ the family of all sets $F$ such that  
	\begin{enumerate}[(i)]
		\item they are union of stripes oriented along the direction $e_i$ 
		\item their connected components of the boundary are distant at least $\eta$. 
	\end{enumerate}
	We denote by 
	\begin{equation} 
		\label{eq:defDEta}
		\begin{split}
			D^{i}_{\eta}(E,Q) := \inf\Big\{ \frac{1}{\vol(Q)} \int_{Q} |\chi_{E} -\chi_{F}|:\ F\in \Acal^{i}_{\eta} \Big\} \quad\text{and}\quad D_{\eta}(E,Q) = \inf_{i} D^{i}_{\eta}(E,Q).
		\end{split}
	\end{equation} 
	Finally, we let $\mathcal A_\eta:=\cup_{i}\mathcal A^i_{\eta}$.
\end{definition}

We recall also the following properties of the functional defined in~\eqref{eq:defDEta}.

\begin{remark} The distance function from the set of stripes satisfies the following properties.
	\label{rmk:lip} \ 
	\begin{enumerate}[(i)]
		\item Let $E\subset\R^d$.  Then the map $z\mapsto D_{\eta}(E,Q_{l}(z))$ is Lipschitz, with Lipschitz constant $C_d/l$, where $C_d$ is a constant depending only on the dimension $d$. 
		
		In particular, whenever $D_{\eta}(E,Q_{l}(z)) > \alpha$ and $D_{\eta}(E,Q_{l}(z')) < \beta$,  then $|z - z'|> l(\alpha - \beta)/C_{d}$.

		\item
		For every $\varepsilon$ there exists ${\delta}_0= \delta_0(\varepsilon)$ such that  for every $\delta \leq \delta_0 $ whenever $D^{j}_{\eta}(E,Q_{l}(z))\leq \delta$ and $D^{i}_{\eta}(E,Q_{l}(z))\leq \delta$ with $i\neq j$ for some $\eta>0$,  it holds 
		\begin{equation*}
			\begin{split}
				\min\big(|Q_l(z)\setminus E|, |E \cap Q_l(z)| \big) \leq\varepsilon. 
			\end{split}
		\end{equation*}
	\end{enumerate}
\end{remark}

\section{Key features of the one-dimensional problem}
\label{sec:One-Dimensional}

In this section we study properties of the energy functional  $\Fcal_{\tau, L}$ on one-dimensional sets, namely sets belonging to the set $\mathcal C_{L}$ defined in \eqref{eq:cl}, and of its minimizers. The main result of this section, which will be used in the proof of Theorem \ref{thm:main} is Theorem \ref{cor:conv}.

As recalled in Section \ref{sec:introduction}, letting $e_{\infty,\tau}$ be the minimal value obtained by $\Fcal_{\tau, L}$ on the class of sets $\mathcal C_L$ as $L$ varies (see \eqref{eq:einfty}) and 
\begin{equation*}
	E_{h}=\bigcup_{k\in\Z}[2kh, (2k+1)h)\times\R^{d-1},
\end{equation*} 
 one has the following 

\begin{theorem}[\cite{dr_arma}Theorem 1.1]\label{thm:emin}
	There exists $\bar{\tau}_0>0$ such that for all $0\le\tau\leq\bar{\tau}_0$ there exists a unique $h^*_\tau>0$ such that
	\begin{equation}\label{eq:einftyeq}
		e_{\infty,\tau}=\Fcal_{\tau,2h^*_\tau}(E_{h^*_\tau})=\inf_h\Fcal_{\tau,2h}(E_{h}).
	\end{equation} 
\end{theorem}

 The existence of at least one finite optimal period for the one-dimensional problem  has been proved in \cite{chen2005periodicity, muller1993singular, ren2003energy} using convexity arguments and in \cite{giuliani06_ising_model_with_long_range} using reflection positivity techniques. The uniqueness of the optimal period for sufficiently small $\tau$ has been proved in \cite{dr_arma} for a slightly more general class of potentials with power law scaling and reflection positivity properties. 
 
 Let now for simplicity of notation define
 \begin{equation}
 	e_\tau(h)=\Fcal_{\tau, 2h}(E_{h}).
 \end{equation}
 In particular, by Theorem \ref{thm:emin} 
 \[
 e_{\infty,\tau}=e_\tau(h^*_\tau)=\inf_he_\tau(h).
 \]
 As computed in \cite{gr,ker}, one has the following formula for the energy of stripes of width and distance $h$
 \begin{equation}\label{eq:enform1}
 	e_\tau(h)=-\frac1h+\frac{C(\tau^{1/\beta}/h)}{h^{q-1}},
 \end{equation} 
where 
\begin{equation}\label{eq:enform2}
	C(s)=\frac{2C_1}{(q-1)(q-2)}\Bigl\{\sum_{k\geq 0}\frac{2}{(2k+1+s)^{q-2}}-\frac{2}{(2k+2+s)^{q-2}}\Bigr\}
\end{equation}
and
\begin{equation*}
	C_1=\int_{\R^{d-1}}\frac{1}{(\|\xi\|_1+1)^p}\d\xi.
\end{equation*}

For a complete and detailed proof of \eqref{eq:enform1} and \eqref{eq:enform2} we refer to \cite{gr, ker}.
 
Before stating and proving Theorem \ref{cor:conv} we need a series of preliminary results.
 
 \begin{lemma}\label{lemma:convint}
 	Let $0<\eps\ll1$. There exists $\bar{\tau}_1>0$ and $0<\bar c_1<\bar c_2$, $\bar c_3>0$ such that for all $0\leq\tau\leq\bar{\tau}_1$ and for all $h>0$ such that
 	\begin{equation}\label{eq:etaue0}
 	e_\tau(h)\leq e_0(h^*_0)+\eps
 	\end{equation}
 	it holds
 	\begin{align}
 		&\bar c_1\leq h\leq\bar c_2\label{eq:c1c2}\\
 		&\partial^2_he_\tau(h)\geq \bar c_3\label{eq:c3}
 	\end{align}
 \end{lemma}
The above lemma implies that whenever for $\tau$ sufficiently small the energy of some periodic stripes of period $h$ is close to the minimal energy for $\tau=0$ as in  \eqref{eq:etaue0} (which we will see it is the case for the stripes of minimal period in Lemma \ref{lemma:elemma}), then $h$ must lie in some given interval \eqref{eq:c1c2} on which the function $e_\tau$ is strictly convex  \eqref{eq:c3}.

\begin{proof}
	\textbf{Step 1} First of all, we prove Lemma \ref{lemma:convint} for $\tau=0$.

	By explicit computations, one can see that the unique minimum and stationary point of $e_0$ is given by
	\begin{equation*}
		h^*_0=\Bigl((q-1)C(0)\Bigr)^{\frac{1}{q-2}}.
	\end{equation*}
Moreover, from the explicit formulas \eqref{eq:enform1} and \eqref{eq:enform2} one can directly see that there exist $0<\bar c_1<\bar c_2$ such that
\begin{equation}\label{eq:e0ineq}
	e_0(h)\leq e_0(h^*_0)+\frac32{\eps}\quad\Rightarrow\quad \bar c_1\leq h\leq \bar c_2.
\end{equation}
Indeed, 
	\begin{align*}
	e_0(h^*_{0})&=-\frac{q-2}{(q-1)^{\frac{(q-1)}{(q-2)}}}C(0)^{-\frac{1}{q-2}}=-\bar C C(0)^{-\frac{1}{q-2}},
\end{align*}
thus the left inequality in \eqref{eq:e0ineq} becomes 
\begin{equation}
	-h^{q-2}\leq -C(0)-\Bigl[\bar C C_0^{-\frac{1}{q-2}}-\frac32\eps\Bigr]h^{q-1}.
\end{equation}
Hence, the bounds in the r.h.s. of \eqref{eq:e0ineq} follow from the inequalities $-h^{q-2}\leq -C(0)$ and $-h^{q-2}\leq -\Bigl[\bar C C_0^{-\frac{1}{q-2}}-\frac32\eps\Bigr]h^{q-1}$.

Moreover, choosing eventually $\bar c_1$ and $\bar c_2$ relative to a smaller $\eps$ in \eqref{eq:e0ineq}, there exists  $\tilde c_3>2\bar c_3>0$ such that  
\begin{equation}\label{eq:d2f0}
	{\bar c_1}\leq h\leq{\bar c_2}\quad\Rightarrow  \quad\tilde c_3\geq\partial^2_h e_0(h)\geq 2\bar c_3.
\end{equation} 

Indeed, one has that
\begin{align*}
	\partial^2_h e_0(h)=\frac{-2h^{q-2}+q(q-1)C(0)}{h^{q+1}}.
\end{align*}
and 
\begin{equation*}
	(h^*_0)^{q-2}=(q-1)C(0).
\end{equation*}
In particular,
\begin{equation*}
	\partial^2_h e_0(h^*_{0})={(q-1)^{-3/(q-2)}(q-2)C(0)^{-3/(q-2)}}.
\end{equation*}
Hence also~\eqref{eq:d2f0} holds.

\textbf{Step 2} Let us now estimate the difference between  $e_\tau$ and $e_0$ and between their derivatives for small values of $\tau$. 
First of all we claim that  there exist $\bar c_4>0$ and $\tilde\tau>0$ such that $\bar c_4<h^*_{0}$ and for all $0\leq\tau\leq\tilde\tau$ one has that 
\begin{equation}\label{eq:c4}
	e_\tau(h)<0\quad\Rightarrow\quad h\geq{\bar c_4}.
\end{equation}
Moreover, by comparing the expression for $C(\tau^{1/\beta}/h)$ with the one for $C(0)$ one has that there exist $\bar c_5,\bar c_6,\bar c_7>0$ such that for all $h\geq \bar c_4$ it holds 
\begin{align}
	\big|e_\tau(h)-e_0(h)\big|&\leq \bar c_5\tau^{1/\beta}\label{eq:d0}\\
	\big|\partial_he_\tau(h)-\partial_he_0(h)\big|&\leq \bar c_6\tau^{1/\beta},\label{eq:d1}\\
	\big|\partial^2_he_\tau(h)-\partial^2_he_0(h)\big|&\leq \bar c_7\tau^{1/\beta}.\label{eq:d2}
\end{align}
While \eqref{eq:c4} and \eqref{eq:d0} follow from the formula for $e_\tau(h)$, 
in the proof of \eqref{eq:d1} and \eqref{eq:d2} one uses the formulas
\begin{equation}\label{eq:dere}
	\partial_he_\tau(h)=\frac{1}{h^2}-(q-1)\frac{C(\tau^{1/\beta}/h)}{h^q}-\Bigl(\frac{\tau^{1/\beta}}{h}\Bigr)\frac{\partial_sC(\tau^{1/\beta}/h)}{h^q}
\end{equation}
and 
	\begin{align}
		\partial^2_he_\tau(h)&=-\frac{2}{h^3}+q(q-1)\frac{C(\tau^{1/\beta}/h)}{h^{q+1}}\notag\\
		&+(q-1)\Bigl(\frac{\tau^{1/\beta}}{h}\Bigr)\frac{\partial_s C(\tau^{1/\beta}/h)}{h^{q+1}}+\Bigl(\frac{\tau^{1/\beta}}{h}\Bigr)^2\frac{\partial_s^2C(\tau^{1/\beta}/h)}{h^{q+1}}\notag\\
		&+q\Bigl(\frac{\tau^{1/\beta}}{h}\Bigr)\frac{\partial_sC(\tau^{1/\beta}/h)}{h^{q+1}}.\label{eq:d2etau}
		\end{align}

\textbf{Step 3}
We are now ready to prove \eqref{eq:c1c2}. 
	Let $0<\eps\ll1$ to be fixed later. By~\eqref{eq:c4} and~\eqref{eq:d0}, there exists $0<\check\tau$ such that whenever $0\leq\tau\leq\check \tau$ and $e_\tau(h)<0$ then $|e_\tau(h)-e_0(h)|\leq\eps/2$. In particular,
	\begin{equation*}
		e_\tau(h)\leq e_0(h^*_{0})+{\eps}\quad\Rightarrow\quad e_0(h)\leq e_0(h^*_{0})+\frac32{\eps}.
	\end{equation*}
	Thus, by~\eqref{eq:e0ineq}, one has that $\bar c_1\leq h\leq \bar c_2$.
	Provided $\eps,\check\tau$ are sufficiently small, one can assume that
	\[
	{\bar c_4}\leq{\bar c_1}\leq h\leq{\bar c_2}.
	\]
	Using~\eqref{eq:d2f0} and~\eqref{eq:d2} one obtains
	\begin{align}
		\partial_h^2e_\tau(h)&\geq	\partial_h^2e_0(h)-	|\partial_h^2e_\tau(h)-	\partial_h^2e_0(h)|\notag\\
		&\geq2\bar  c_3-\bar c_7\check\tau^{1/\beta}\notag\\
		&\geq\bar c_3>0
	\end{align}
	In particular, one has the strict convexity of $e_\tau$ on the region where the minimizers are concentrated, thus proving \eqref{eq:c3}.

\end{proof}

 From \cite{gr} we recall also the following result.

\begin{theorem}\label{T:1d} 
	There exists $C>0$ and $\bar{\tau_2}>0$ such that for every $0\leq\tau<\bar{\tau}_2$ and for every $L>0$, the minimizers of $\Fcal_{\tau,L}$ in the class $\mathcal C_L$ are periodic stripes of period $h_{\tau,L}$ for some (possibly non-unique) $h_{\tau,L}>0$ satisfying
	\begin{equation}\label{eq:hC}
	|h_{\tau,L}-h_\tau^*|\leq \frac {C} L.
	\end{equation}
\end{theorem}

We collect  also two useful facts  of immediate proof in the following lemma. 

\begin{lemma}\label{lemma:elemma}
	One has the following:
	\begin{enumerate}
		\item 	\begin{equation}
		\lim_{L\to+\infty}e_\tau(h_{\tau,L})=\inf_L\inf_{h\in L/2\N}e_\tau(h)=e_\tau(h^*_\tau)
	\end{equation}
\item For any $\eps>0$ and $c>0$ there exists $\bar{\tau}_3>0$ such that for all $0\leq\tau\leq\bar {\tau}_3$ it holds
\begin{equation}
	e_\tau(h^*_\tau)\leq e_0(h^*_0)+c\eps.
\end{equation}
\end{enumerate}
\end{lemma}

As a direct consequence of Theorem \ref{T:1d}, of  Lemma \ref{lemma:elemma} and Lemma \ref{lemma:convint}, one obtains  the following 
\begin{theorem}
	\label{cor:conv}
	Let $0<\eps\ll1$. There exist $\bar{\tau}_4\leq\min\{\bar{\tau}_0,\dots,\bar{\tau}_3\}$, $\bar L>0$ such that for all $0\leq\tau\leq\bar \tau_4$ and for all $L\geq\bar L$
	\begin{equation}
		\bar c_1<h_{\tau,L}<\bar c_2\quad\text{ and }\quad \partial_h^2e_\tau(h_{\tau,L})\geq\bar c_3>0.
	\end{equation}
	where $\bar c_1,\bar c_2$ and $\bar c_3$ are defined in Lemma \ref{lemma:convint}.

	Moreover, the following holds:
	\begin{equation}\label{eq:epsineq}
		|\partial_he_\tau(h_{\tau,L})|\leq\eps, \qquad|e_{\tau}(h_{\tau,L})-e_\tau(h^*_\tau)|\leq\frac{\eps C}{L},
	\end{equation}
where $C$ is the constant appearing in \eqref{eq:hC}.
	
\end{theorem}

\section{Preliminary lemmas}

In this section we collect a series of Lemmas and Propositions which will be used in the proof of Theorem~\ref{thm:main}. In the one-dimensional optimization Lemma~\ref{lemma:1D-optimization} and in Lemma~\ref{lemma:stimaLinea} we will have now to take into account the minimal energy density for periodic sets of period compatible with the length of the interval on which the optimization takes place. 

At this aim we introduce the following notation: for all intervals $I\subset\R$, we define
\[
h_\tau(I)=\argmin \bigl\{e_{\tau}(h):h\in|I|/(2\N)\bigr\}.
\]
In general, $h_\tau(I)$ might contain different periods $h$ giving all the same energy $e_\tau(h)$. Whenever, with a slight abuse of notation, we will be speaking of $h_\tau(I)$ as if it were a single period is because the properties of all the periods contained in $h_\tau(I)$ are in that case equivalent (meaning than that all $h\in h_\tau(I)$ have the same property).

We start with recalling the following lemma, corresponding to Remark 7.1 in~\cite{dr_arma}. The term $r_{i,\tau}$ penalizes small  sets $E$ whose  one-dimensional slices in direction $e_i$ have boundary points which are close to each other. This is expressed quantitatively by the estimate \eqref{eq:stimamax1_eq}. 

\begin{lemma}
	\label{rmk:stimax1}
	There exist $\eta_0 > 0$ and  $\tau_{0} > 0$ such that for every  $0<\tau< \tau_0$, whenever  $E\subset \R^d$ and $s^-<s<s^+\in \partial E_{t_{i}^{\perp}}$ are three consecutive points satisfying  $\min(|s - s^- |,|s^+ -s |) <\eta_0$, then $r_{i,\tau}(E,t_{i}^{\perp},s) > 0$. 
	
	In particular, the following estimate holds 
	\begin{equation}
		\label{eq:stimamax1_eq}
		\begin{split}
			r_{i,\tau}(E,t^{\perp}_{i},s) \geq -1 + C_{1}C_{2} \min(|s-s^+ |^{-\beta},\tau^{-1}) + C_{1}C_{2}\min(|s-s^-|^{-\beta} , \tau^{-1})
		\end{split}
	\end{equation}
	where $C_{1}=\int_{\R^{d-1}}\frac{1}{(\|\xi\|_1+1)^p}\d\xi$ and $C_{2}=\frac{1}{(q-1)(q-2)}$. 
	
	\end{lemma}


When dealing with one-dimensional optimizations which are independent of the fact that the one-dimensional sets to which they are applied are slices of a $d$-dimensional set, we will use the following one-dimensional analogue of~\eqref{eq:ritau}. Given $E\subset \R$  a set of locally finite perimeter and let $s^-, s,s^+\in \partial E$, one defines
\begin{equation}\label{eq:rtau1D}
	\begin{split}
		r_{\tau}(E,s) := -1 & + \int_\R |\rho| \widehat{K}_{\tau}(\rho)\d\rho  -  \int_{s^-}^{s} \int_0^{+\infty}  |\chi_{E}(\rho+ u) - \chi_{E}(u)| \widehat{K}_{\tau} (\rho)\d\rho  \du \\ & - \int_{s}^{s^+} \int_{-\infty}^0  |\chi_{E}(\rho+ u) - \chi_{E}(u)| \widehat{K}_{\tau} (\rho)\d\rho  \du. 
	\end{split}
\end{equation}

The quantities defined in~\eqref{eq:ritau} and~\eqref{eq:rtau1D} are related via $r_{i,\tau}(E,t^\perp_i,s) = r_{\tau}(E_{t^\perp_{i}},s)$.

In the next Lemma we recall Lemma 7.5 in~\cite{dr_arma}, containing  a lower bound for the first term of the decomposition~\eqref{eq:decomposition} as $\tau\to0$. Thanks to the inequality \eqref{eq:gstr5} the penalization of close boundary points for a family of  sets $E_\tau$ is preserved in the limit as $\tau\to0$.

\begin{lemma}
	\label{lemma:technicalBeforeLocalRigidity}
	Let $E_{0}, \{E_{\tau}\}\subset \R$  be a family of sets of locally finite perimeter and $I\subset \R$ be an open bounded interval.   Moreover, assume that $E_{ \tau}\to E_{0}$ in $L^1(I)$.  
	If we denote by $\{k^{0}_{1},\ldots,k^{0}_{m_{0}}\} = \partial E_{0}\cap I $, then
	\begin{equation}
		\label{eq:gstr5}
		\liminf_{\tau\downarrow 0}\sum_{\substack{s\in \partial E_{\tau}\\ s\in I}}r_{\tau}(E_{\tau},s) \geq \sum_{i=1}^{m_{0}-1}(-1 + {C_{1}C_{2}}|k^{0}_{i} - k^{0}_{i+1} |^{-1}),
	\end{equation}
	where $r_{\tau}$ is defined in~\eqref{eq:rtau1D}.
\end{lemma}

The next proposition contains the main symmetry breaking result at mesoscopic scale $l$: on a square of size $l$, if $\tau$ is  sufficiently close to $0$, a bound on the energy corresponds to a bound on the $L^1$-distance to the unions of stripes.  
It corresponds to Lemma 7.6 in~\cite{dr_arma}. In the limit as $\tau\to0$, for any $p\geq d+2$  sets of bounded energy have to be exactly stripes, via a rigidity argument that uses \eqref{eq:gstr5} and a lower bound on the cross interaction term $w_\tau$.

\begin{proposition}[Local Rigidity] 
	\label{lemma:local_rigidity_alpha}
	For every $M > 1,l,\delta > 0$, there exist $\tau_1>0$ and $\bar{\eta} >0$ 
	such that whenever $0<\tau< {\tau}_1$  and $\bar F_{\tau}(E,Q_{l}(z)) < M$ for some $z\in [0,L)^d$ and $E\subset\R^d$ $[0,L)^d$-periodic, with $L>l$, then it holds $D_{\eta}(E,Q_{l}(z))\leq\delta$ for every $\eta < \bar{\eta}$. Moreover $\bar{\eta}$ can be chosen independently  of $\delta$.  Notice that ${\tau}_1$ and $\bar{\eta}$ are independent of $L$.
\end{proposition} 

In particular, one has the following $\Gamma$-convergence result.
\begin{corollary}\label{cor:gammaconv}
	Let $0<\tau\ll1$. One has that the following holds:
	\begin{itemize}
		\item Let $\{E_{\tau}\}$ be a sequence such that $\sup_{\tau} \bar{F}_{ \tau}(E_{\tau}, Q_l(z)) < \infty$. 
		Then  as $tau\to0$ the sets $E_{\tau}$ converge in $L^1$  up to  subsequences to some set $E_{0}$ of finite perimeter and 
		\begin{align}
			\liminf_{\tau \rightarrow 0} \bar{F}_{ \tau}(E_{\tau}, Q_l(z)) \geq \bar{F}_{0}(E_{0}, Q_l(z)) . 
			\label{eq:liminfLocalGamma}
		\end{align}
		\item For every set $E_{0}$ with $\bar{F}_{0}(E_{0}, Q_l(z)) < + \infty$, there exists a sequence $\{E_{\tau}\}$ converging in $L^1$ to $E_{0}$ as $\tau\to0$ and such that
		\begin{align}
			\limsup_{\tau \rightarrow 0} \bar{F}_{ \tau}(E_{\tau}, Q_l(z)) = \bar{F}_{0}(E_{0}, Q_l(z)) . 
		\end{align} 
	\end{itemize}
\end{corollary}

The following local stability proposition corresponds to Lemma 7.8 in~\cite{dr_arma}. Roughly speaking,  it shows that whenever a set $E\subset\R^d$ is $L^1$-close to a set $S$ which is a union of stripes with boundaries orthogonal to  $e_i$ in a certain cube, then it is not energetically convenient for the set $E$ to have non-straight boundaries in direction $e_j$ with $j\neq i$ (namely, to deviate from being exactly stripes with boundaries orthogonal to $e_i$)
Indeed, in such a case either the local contribution given by $r_{i,\tau}$ or the one given by the cross interaction term  $v_{i,\tau}$ are large.

\begin{proposition}[Local Stability]
	\label{lemma:stimaContributoVariazionePiccola}
	Let  $(t^{\perp}_{i}+se_i)\in (\partial E) \cap [0,l)^d$,  and  $\eta_{0}$, $\tau_0$ as in Lemma~\ref{rmk:stimax1}. Then there exist ${\tau_2}\leq\tau_0$ and  $\varepsilon_2$ (independent of $l$) such that for every $0<\tau < {\tau_2}$, and $0<\varepsilon < {\varepsilon_2}$ the following holds: assume that 
	\begin{enumerate}[(a)]
		\item $\min(|s-l|, |s|)> \eta_0$ (i.e. the boundary point $s$ in the slice of $E$ is sufficiently far from the boundary of the cube)
		\item $D^{j}_{\eta}(E,[0,l)^d)\leq\frac {\varepsilon^d} {16 l^d}$ for some $\eta> 0$ and  with $j\neq i$ (i.e. $E\cap [0,l)^d$ is close to stripes with boundaries orthogonal to $e_j$ for some $j\neq i$)
	\end{enumerate}
	Then 
	\[r_{i,\tau}(E,t^{\perp}_{i},s) + v_{i,\tau}(E,t^{\perp}_{i},s) \geq 0.\] 
\end{proposition}

As shown in \cite{ker}, the above stability argument can be extended to all $p>d+1$, provided $\tau$ is sufficiently small depending on $p$.

In the following lemma we show  the main one-dimensional estimate needed in the proof  of Lemma~\ref{lemma:stimaLinea}. Roughly speaking, it shows that up to an error term (i.e. the constant $C_0$ in \eqref{eq:gstr40}), the contribution of the one-dimensional term $r_{\tau}(E,s)$ on an interval $I$ is bounded from below by the minimal energy density for periodic sets of period $|I|$. In particular, the estimate below takes into account minimal energy relative to the length of the interval on which it is performed, thus differing from previous optimization estimates obtained in \cite{dr_arma} in which the global minimum for the energy density over all possible periods was considered. Notice that it is valid for sufficiently large intervals. This will not be a restriction since we are interested in proving optimality of striped patterns in the large volume limit. 

\begin{lemma}
	\label{lemma:1D-optimization}
	There exists $C_0>0$ such that the following holds.
	Let $E\subset \R$  be a set of locally finite perimeter and $I\subset \R$ be an open interval such that $|I|>\bar L$ and $\bar L$ is as in Theorem \ref{cor:conv}. 
	Let $r_{\tau}(E,s)$ be defined as in~\eqref{eq:rtau1D}.
	Then for all $0<\tau< \min\{\tau_0,\bar{\tau}_4\}$, where $\tau_0$ is given in Lemma~\ref{rmk:stimax1} and $\bar{\tau}_4$ is given in Theorem \ref{cor:conv}, it holds
	\begin{equation}
		\label{eq:gstr40}
		\sum_{\substack{s \in \partial E\\ s \in I}} r_{\tau}(E,s) \geq |I|e_\tau(h_\tau(I)) - C_0.
	\end{equation} 
\end{lemma}

\begin{proof}
	Let us denote by $k_1< \ldots< k_m $ the  points of $\partial E \cap I$, and 
	
	\begin{equation*}
		k_0 = \sup\{ s\in \partial E: s < k_1\} \qquad\text{and}\qquad 
		k_{m+1} = \inf\{ s\in \partial E: s > k_m\} 
	\end{equation*}
	
	W.l.o.g. we may assume that  $r_{\tau}(E,k_1) < 0$ and that $r_{\tau}(E,k_m) < 0$. 
	
	We claim that if this is not the case one can consider $I' \subset I$  such that $r_{\tau}(E,k'_1) < 0$ and $r_{\tau}(E,k'_{m'}) < 0$, where  $k'_1,\cdots,k'_{m'}$ are the points of $\partial E \cap I'$.
	Indeed, if estimate \eqref{eq:gstr40} holds for $I'$ then one has the following chain of inequalities
	\begin{equation}\label{eq:ii'}
		\sum_{\substack{s \in \partial E\\ s \in I}}
r_{\tau}(E,s) \geq \sum_{\substack{s \in \partial E\\ s \in I'}} r_{\tau}(E,s) \geq e_{\tau}(h_\tau(I')) |I'| - C_0.
\end{equation}
Moreover, by Theorem \ref{cor:conv}, there exists $h\in h_\tau(\bar I)$  for some interval $\bar I$ with $|\bar I|\geq \bar L$ satisfying $e_\tau(h_\tau(\bar I))\leq e_\tau(h_\tau(I'))$.
Identifying with a slight abuse of notation such $h$ with $h_\tau(\bar I)$, by Theorem \ref{T:1d} and Theorem  \ref{cor:conv} one has that
\begin{equation}\label{eq:htau2}
	h_\tau(\bar I), \, h_\tau(I)\,\in [\bar c_1,\bar c_2], \bigl|h_\tau(\bar I)-h_\tau(I)\bigr|\leq\frac{C}{\min\{|I|, |\bar I|\}}
\end{equation}
and 
\begin{equation}
	\partial_h^2e_\tau(h)\geq\bar c_3>0\quad\text{on $[\bar c_1,\bar c_2]$}.
\end{equation}
Then one has that
\begin{align}
	e_\tau(h_\tau(I')) |I'|&\geq e_\tau(h_\tau(\bar I))|I'|\geq e_\tau(h_\tau(I)) |I'|+(h_\tau(\bar I)-h_\tau(I))\partial_h e_\tau(h_\tau(I))|I'|\notag\\
	&\geq e_\tau(h_\tau(I)) |I|+(h_\tau(\bar I)-h_\tau(I))\partial_h e_\tau(h_\tau(I))|I'|,\label{eq:27.5}
\end{align}
where in the last inequality we used the fact that $|I|\geq |I'|$ and that $e_\tau(h_\tau(I))<0$. 
Now observe that by formula \eqref{eq:dere} on the interval $[\bar c_1,\bar c_2]$ of Theorem \ref{cor:conv} the function $\partial_he_\tau(h)$ is uniformly bounded, i.e. $|\partial_he_\tau(h)|\leq\bar C$. This fact together with \eqref{eq:htau2} implies the following 
\begin{align}
	e_\tau(h_\tau(I')) |I'|&\geq e_\tau(h_\tau(I)) |I|+(h_\tau(\bar I)-h_\tau(I))\partial_h e_\tau(h_\tau(I))|I'|\notag\\
	&\geq e_\tau(h_\tau(I)) |I|-\frac{C}{\min\{|I|, |\bar I|\}}{\bar C}|I'|\label{eq:lineq1}\\
	&\geq e_\tau(h_\tau(I)) |I|- C\bar C,\label{eq:lineq2}
\end{align}
where in passing from \eqref{eq:lineq1} to \eqref{eq:lineq2} we used the fact that $|I'|<|I|$ and that $\bar I$ with optimal compatible period $h_\tau(\bar I)$ can be chosen to be such that $|\bar I|=|I|+O(1)$. 

Thus from \eqref{eq:ii'} and \eqref{eq:lineq2} it follows that, eventually enlarging the constant $C_0$ in \eqref{eq:gstr40}, the main estimate \eqref{eq:gstr40} is valid also for the interval $I$ whenever it is valid for $I'$.

Because of Lemma~\ref{rmk:stimax1}, the fact that $r_{\tau}(E,k_1) < 0$ and $r_{\tau}(E,k_m) < 0$ implies that there exists $\eta_0>0$ (for all  $\tau \leq \tau_0$) such that 
\begin{equation*}
\min(|k_1 - k_0 |, |k_2 - k_1 |, |k_{m-1} - k_m |,|k_{m+1} - k_m |) > \eta_0. 
\end{equation*}

We claim that 
\begin{equation}
\label{eq:gstr4}
\sum_{i =1}^{ m } r_{\tau}(E,k_i) \geq \sum_{i =1}^{ m } r_{\tau}(E',k_i)  - \bar C_0 
\end{equation}
where $E'$ is obtained by extending periodically $E$ with the pattern contained in $E \cap (k_1,k_m)$ and $\bar C_0  = \bar C_0(\eta_0) > 0$. 
The construction of $E'$ can be done as follows: if $m$ is odd we repeat periodically $E\cap (k_1,k_m)$, and if $m$ is even we repeat periodically $(k_1-\eta_{0}, k_m)$. 

Thus we have constructed a set $E'$ which is periodic of period $k_m-k_1$ or $k_m- k_1 + \eta_0$. Therefore, setting 

\begin{equation}
\label{eq:ggstr1}
\sum_{i=1}^m r_\tau(E',k_i) \geq e_{\tau} (h_\tau(I)) - \tilde C _0,
\end{equation}
where $ \tilde C_0 = \tilde C_0(\eta_0)$. 
Inequality~\eqref{eq:ggstr1} follows by definition of optimal energy density relative to the interval $I=(k_1,k_m)$.

Inequality~\eqref{eq:ggstr1} combined with \eqref{eq:gstr4} yields \eqref{eq:gstr40}.

To show \eqref{eq:gstr4}, notice that the symmetric difference between $E$ and $E'$ satisfies
\begin{equation*}
E \Delta E'  \subset (- \infty, k_1 - \eta_0) \cup (k_m+\eta_0, + \infty),
\end{equation*}
where $\eta_{0}$ is the constant defined in Lemma~\ref{rmk:stimax1}.
To obtain \eqref{eq:gstr4},  we need to estimate $|\sum_{i=1}^m r_\tau(E,k_i) - \sum_{i=1}^m r_\tau(E',k_i) |$. 
Let 
\begin{equation*}
\begin{split}
	\sum_{i=1}^m r_\tau(E,k_i) - \sum_{i=1}^m r_\tau(E',k_i) =
	I_1 + I_2,
\end{split}
\end{equation*}
where
\begin{equation*}
\begin{split}
	I_1 =  \sum_{i=0}^{m-1}&  \int_{k_{i}}^{k_{i+1}}\int_{0}^{+\infty} \bigl[\bigl(s - |\chi_{E}(s+u) - \chi_{E}(u)|\bigr)-\bigl(s - |\chi_{E'}(s+u) - \chi_{E'}(u)|\bigr)\bigr]\widehat{K}_{\tau}(s) \ds \du\\
	I_2=  \sum_{i=1}^{m} & \int_{k_{i}}^{k_{i+1}}\int_{-\infty}^{0} \bigl[\bigl(s - |\chi_{E}(s+u) - \chi_{E}(u)|\bigr)-\bigl(s - |\chi_{E'}(s+u) - \chi_{E'}(u)|\bigr)\bigr]\widehat{K}_\tau(s) \ds \du. 
\end{split}
\end{equation*}
Thus by using the integrability of $\widehat K$, we have that 
\begin{equation*}
\begin{split}
	|I_1| \leq \int_{k_{0}} ^{k_{m}} \int_{0}^{+\infty} \chi_{E \Delta E'}(u +s  ) \widehat{K}_{\tau}(s) \ds \du \leq\int_{k_0}^{k_m}\int_{k_{m} + \eta_0}^{\infty} \widehat{K}_{\tau}(u-v) \dv\du \leq  \frac{C_{0}}{2},
\end{split}
\end{equation*}
where $C_{0}$ is a constant depending only on $\eta_0$.  Similarly, $|I_2| \leq C_{0}/2$

Thus we have that
\begin{equation*}
\Big|\sum_{i=1}^m r_\tau(E,k_i) - \sum_{i=1}^m r_\tau(E',k_i)\Big| \leq C_0.
\end{equation*}

\end{proof}

The next lemma is the analogue of Lemma 7.11 in~\cite{dr_arma} and gives a lower bound on the energy in the case almost all the volume of $Q_l(z)$ is filled by $E$ or $E^c$ (this will be the case on the set $A_{-1}$ defined in~\eqref{a1}).
\begin{lemma}
	\label{lemma:stimaQuasiPieno}
	Let  $E$ be a set of locally finite perimeter  such that $\min(|Q_{l}(z)\setminus E|, |E\cap Q_{l}(z) |)\leq {\delta} l^d$, for some $\delta>0$. Then 
	\begin{equation*}
		\begin{split}
			\bar F_{\tau} (E,Q_{l}(z)) \geq -\frac {\delta d } {\eta_0 },
		\end{split}
	\end{equation*}
	where $\eta_0$ is defined in Lemma~\ref{rmk:stimax1}.
\end{lemma}

The following lemma contains the main lower bounds of the complete functional along one-dimensional slices. It relies on the previous lemmas of this section. In our setting, namely aiming at proving striped pattern formation in the large volume limit along periodic boxes of arbitrary size, the optimal energy densities relative to intervals of different length have to be taken into account. In order to exploit the validity of the one-dimensional estimate of Lemma \ref{lemma:1D-optimization} the mesoscopic scale $l$ has to be sufficiently large. 

\begin{lemma}
	\label{lemma:stimaLinea}
	Let  ${\eps_2},{\tau_2}>0$ as in Lemma~\ref{lemma:stimaContributoVariazionePiccola}, $\bar{\tau}_4$ as in Theorem \ref{cor:conv} and let $l\geq 2\bar L$ with $\bar L$ as in Theorem \ref{cor:conv}. Let $\delta=\eps^d/(16l^d)$ with $0<\eps\leq{\eps_2}$, $0<\tau\leq\min\{\tau_2,\bar{\tau}_4\}$ and 
	$C_0$ be the constant appearing in Lemma~\ref{lemma:1D-optimization}. Let $t_i^\perp\in[0,L)^{d-1}$ and $\eta>0$.

	The following hold: there exists a constant $C_1$ independent of $l$ (but depending on the dimension and on $\eta_0$ as in Lemma~\ref{rmk:stimax1}) such that
	\begin{enumerate}[(i)]
		\item Let $J\subset \R$ an interval  such that for every $s\in J$ one has that  $D^{j}_{\eta}(E,Q_{l}(t^{\perp}_{i}+se_i))\leq \delta$ with $j\neq i$. 
		Then
		\begin{equation}
			\label{eq:gstr20}
			\begin{split}
				\int_{J} \bar{F}_{i,\tau}(E,Q_{l}(t^{\perp}_{i}+se_i))\ds \geq - \frac{C_1}{l}.
			\end{split}
		\end{equation}
		Moreover, if $J = [0,L)$, then 
		\begin{equation}
			\label{eq:gstr21}
			\begin{split}
				\int_{J} \bar{F}_{i,\tau}(E,Q_{l}(t^{\perp}_{i}+se_i))\ds \geq0.
			\end{split}
		\end{equation}
		\item Let $J = (a,b)\subset \R$. 
		If for $s=a$ and $s=b$ it holds $D_\eta^j(E,Q_{l}(t^{\perp}_i+se_i)) \leq \delta$ with $j\neq i$, then  setting 
		\begin{equation}\label{eq:jl}
			J_l=(a+l/4,b-l/4)\text{ if $|b-a|>l$, } 
		\end{equation} 
		one has that
		\begin{equation}
			\label{eq:gstr27}
			\begin{split}
				\int_{J} \bar{F}_{i,\tau}(E,Q_{l}(t^{\perp}_{i}+se_i))\ds \geq\Big(|J_l| e_\tau(h_\tau(J_l)) -C_0\Big)\chi_{(0,+\infty)}(|J|-l)-\frac{C_1} l,
			\end{split}
		\end{equation}
		otherwise
		\begin{equation}
			\label{eq:gstr36}
			\begin{split}
				\int_{J} \bar{F}_{i,\tau}(E,Q_{l}(t^{\perp}_{i}+se_i))\ds \geq \Big(|J_l|e_\tau(h_\tau(J_l))-C_0\Big)\chi_{(0,+\infty)}(|J|-l) - C_1l.
			\end{split}
		\end{equation}
		Moreover, if $J = [0,L)$, then
		\begin{equation}
			\label{eq:gstr28}
			\begin{split}
				\int_{J} \bar{F}_{i,\tau}(E,Q_{l}(t^{\perp}_{i}+se_i))\ds \geq Le_\tau(h_{\tau,L}).
			\end{split}
		\end{equation}
	\end{enumerate}
\end{lemma}

   \begin{proof}The proof of $(i)$ follows from Lemma~\ref{lemma:stimaContributoVariazionePiccola} as in Lemma 7.9 in~\cite{dr_arma}.

Let us now prove $(ii)$.  For simplicity of notation we assume that $J= (0,l')$.  

One has that

   \begin{equation}
      \label{eq:gstr32}
      \begin{split}
         \int_{J} \bar{F}_{i,\tau}(E,Q_{l}(t^{\perp}_i,s)) \ds &         \geq \frac{1}{l^{d-1}} \int_{Q^{\perp}_{l}(t^{\perp}_{i})} \sum_{\substack{s' \in \partial E_{t'^{\perp}_{i}} 
               \\ s'\in (-\frac l2, l' + \frac l2)}}\frac{|Q^{i}_{l}(s')\cap J|}{l} \Big(r_{i,\tau }(E,t'^{\perp}_{i},s') + v_{i,\tau }(E,t'^{\perp}_{i},s')\Big) \dt'^{\perp}_{i}\\
         & =   \frac{1}{l^{d-1}} \int_{Q^{\perp}_{l}(t^{\perp}_{i})} \sum_{\substack{s' \in \partial E_{t'^{\perp}_{i}} 
               \\ s'\in (\frac l4, l'-\frac l4 )}}\frac{|Q^{i}_{l}(s')\cap J|}{l} \Big(r_{i,\tau }(E,t'^{\perp}_{i},s') + v_{i,\tau }(E,t'^{\perp}_{i},s')\Big) \dt'^{\perp}_{i}\\
         & +   \frac{1}{l^{d-1}} \int_{Q^{\perp}_{l}(t^{\perp}_{i})} \sum_{\substack{s' \in \partial E_{t'^{\perp}_{i}} 
               \\ s'\in (-\frac l2, \frac l4 ]\cup [l'-\frac l4,l'+\frac l2)}}\frac{|Q^{i}_{l}(s')\cap J|}{l} \Big(r_{i,\tau }(E,t'^{\perp}_{i},s') + v_{i,\tau }(E,t'^{\perp}_{i},s')\Big) \dt'^{\perp}_{i}
      \end{split}
   \end{equation}
Let us now show \eqref{eq:gstr36}.
   If $l'\leq l$, 
  then  $(t_i'^\perp,s')\in Q_{l}(t'^\perp_i,0)$ 
   or $(t_i'^\perp,s')\in Q_{l}(t'^\perp_i,l')$. 
   If the condition $D^{j}_{\eta}(E,Q_{l}(t'^\perp_i,0)) \leq \delta$ or $D^{j}_{\eta}(E,Q_{l}(t'^\perp_i,l')) \leq \delta$ is missing, then we will estimate $r_{i,\tau}(E,t'^\perp_{i},s') + v_{i,\tau}(E,t'^{\perp}_{i},s')$ from below with $-1$ whenever the neighbouring ``jump'' points are further than $\eta_0$, together with the fact that $\frac{|Q^{i}_{l}(s')\cap J|}{l}\leq 1$. Otherwise $r_{i,\tau} +v_{i,\tau}\geq 0$. 
   Hence, the inequality~\eqref{eq:gstr32} can be estimated by
   \begin{equation*}
   \frac{1}{l^{d-1}} \int_{Q^{\perp}_{l}(t^{\perp}_{i})} \sum_{\substack{s' \in \partial E_{t'^{\perp}_{i}} 
               \\ s'\in (-\frac l2, l'+\frac l2 ]}}\frac{|Q^{i}_{l}(s')\cap J|}{l} \Big(r_{i,\tau }(E,t'^{\perp}_{i},s') + v_{i,\tau }(E,t'^{\perp}_{i},s')\Big) \dt'^{\perp}_{i} \geq - C_1l.
   \end{equation*}

   If instead $l' > l>2\bar L$ we have that, by Lemma \ref{lemma:1D-optimization}, 
   \begin{align*}
   \frac{1}{l^{d-1}} \int_{Q^{\perp}_{l}(t^{\perp}_{i})} \sum_{\substack{s' \in \partial E_{t'^{\perp}_{i}} 
   		\\ s'\in (l/4, l'-l/4)}} r_{i,\tau }(E, t'^{\perp}_{i},s')  \dt'^{\perp}_{i} &\geq    \frac{1}{l^{d-1}} \int_{Q^{\perp}_{l}(t^{\perp}_{i})} |J_l|e_\tau(h_\tau(J_l)) \dt_i'^\perp-{C_0}\\
   	&= |J_l|e_\tau(h_\tau(J_l))-{C_0},
   \end{align*}
where we used the fact that $|J_l|=l'-l/2>\bar L$. 
 
 Hence
   \begin{equation*}
   \begin{split}
   \frac{1}{l^{d-1}} \int_{Q^{\perp}_{l}(t^{\perp}_{i})} \sum_{\substack{s' \in \partial E_{t'^{\perp}_{i}} 
   		\\ s'\in (l/4, l'-l/4)}} r_{i,\tau }(E,t'^{\perp}_{i},s')  \dt'^{\perp}_{i} \geq \big(|J_l|e_\tau(h_\tau(J_l))-{C_0}\big) \chi_{(0, +\infty)}(|J|-l)-C_1l,
   \end{split}
   \end{equation*}
   
   Thus, since  $\frac34\leq\frac{|Q^{i}_{l}(s')\cap J|}{l}\leq 1$ whenever $s'\in (l/4, l'-l/4)$, ~\eqref{eq:gstr36} follows.

   Let us now turn to the proof of~\eqref{eq:gstr27}.
 Given that $D^j_{\eta}(E,Q_{l}(t^{\perp}_{i},0))\leq \delta$ and $D^{j}_{\eta}(E,Q_{l}(t^\perp_i,l')) \leq\delta$ for some $j\neq i$, by Lemma~\ref{lemma:stimaContributoVariazionePiccola} with $\delta=\varepsilon^d/(16l^d)$ we have that
   \begin{equation}
      \label{eq:gstr37}
      \begin{split}
         r_{i,\tau}(E,t'^\perp_i,s') + v_{i,\tau}(E,t'^\perp_i,s') \geq  0 
      \end{split}
   \end{equation}
   whenever $\min(|s' - l'+l/2| ,|s'- l' -l/2 |) \geq \eta_0$  and $(t'^\perp_i,s')\in Q_{l}(t^\perp_i,l')$
   or  $\min(|s'+l/2| ,|s'-l/2 |)\geq \eta_0$  and $(t'^\perp_i,s')\in Q_{l}(t^\perp_i,0)$. 

   Fix $t'^{\perp}_{i}$. Then
   \begin{equation*}
      \begin{split}
          \sum_{\substack{s' \in \partial E_{t'^{\perp}_{i}} 
               \\ s'\in (-\frac l2, \frac l2 )}}\frac{|Q^{i}_{l}(s')\cap J|}{l} & \Big(r_{i,\tau }(E,t'^{\perp}_{i},s') + v_{i,\tau }(E,t'^{\perp}_{i},s')\Big) \geq
         \\  \geq & \sum_{\substack{s' \in \partial E_{t'^{\perp}_{i}} 
               \\ s'\in (-\frac l2, \frac l2 )\\
            \min(|s'+l/2| ,|s'-l/2 |)\geq \eta_0
            }}\frac{|Q^{i}_{l}(s')\cap J|}{l} \Big(r_{i,\tau }(E,t'^{\perp}_{i},s') + v_{i,\tau }(E,t'^{\perp}_{i},s')\Big)
         \\  + & \sum_{\substack{s' \in \partial E_{t'^{\perp}_{i}} 
               \\ s'\in (-\frac l2, \frac l2 )\\
            \min(|s'+l/2| ,|s'-l/2 |)<  \eta_0
            }}\frac{|Q^{i}_{l}(s')\cap J|}{l} \Big(r_{i,\tau }(E,t'^{\perp}_{i},s') + v_{i,\tau }(E,t'^{\perp}_{i},s')\Big) 
      \end{split}
   \end{equation*}
   Thus by using~\eqref{eq:gstr37}, we have that the first term on the \rhs above is positive. To estimate the last term on the \rhs above we notice that $r_{i,\tau} \geq 0 $ whenever the neighbouring points are closer than $\eta_0$ and otherwise $r_{i,\tau}\geq  -1$.
   Moreover, given that  $\frac{|Q^i_l{(s')}\cap J |}{l} < \frac{\eta_0}{l}$ for $s' \in (-l/2,l/2)\cup (l'-l/2,l'+l/2)$, we have that the last term on the \rhs above can be bounded from below by $-C_1/l$. Finally integrating over $t'^\perp_i$ we obtain that
   \begin{equation*}
      \begin{split}
         \frac{1}{l^{d-1}} \int_{Q^{\perp}_{l}(t^{\perp}_{i})} \sum_{\substack{s' \in \partial E_{t'^{\perp}_{i}} 
               \\ s'\in (-l/2, l/2 )\cup (l'-l/2,l'+l/2)}}\frac{|Q^{i}_{l}(s')\cap J|}{l} \Big(r_{i,\tau }(E,t'^{\perp}_{i},s') + v_{i,\tau }(E,t'^{\perp}_{i},s')\Big) \dt'^{\perp}_{i} \geq -\frac{C_{1}}{l}.
      \end{split}
   \end{equation*}

   By using the above inequality in~\eqref{eq:gstr32} and  the fact  that for every $s'\in (l/2,l'-l/2)$ it holds $\frac{|Q_{l}(s')\cap J |}{l} =1 $, we have that
   \begin{equation*}
      \begin{split}
         \int_{J} \bar{F}_{i,\tau}(E,Q_{l}(t^{\perp}_i,s)) \ds 
         & \geq   \frac{1}{l^{d-1}} \int_{Q^{\perp}_{l}(t^{\perp}_{i})} \sum_{\substack{s' \in \partial E_{t'^{\perp}_{i}} 
               \\ s'\in (l/4, l'-l/4 )}} \Big(r_{i,\tau }(E,t'^{\perp}_{i},s') + v_{i,\tau }(E,t'^{\perp}_{i},s')\Big) \dt'^{\perp}_{i}
         - \frac{C_{1}}{l}
      \end{split}
   \end{equation*}

   To conclude the proof of~\eqref{eq:gstr27}, as for~\eqref{eq:gstr36}, we notice that
   \begin{equation*}
      \begin{split}
         \frac{1}{l^{d-1}} \int_{Q^{\perp}_{l}(t^{\perp}_{i})} \sum_{\substack{s' \in \partial E_{t'^{\perp}_{i}} 
               \\ s'\in (l/4, l'-l/4)}} r_{i,\tau }(E,t'^{\perp}_{i},s')  \dt'^{\perp}_{i} \geq \big(|J_l|e_\tau(h_\tau(J_l))-{C_0}\big) \chi_{(0, +\infty)}(|J| -l),
      \end{split}
   \end{equation*}
   where in the last inequality we have used Lemma~\ref{lemma:1D-optimization} for $E=E_{t'^{\perp}_{i}}$, $J_l=(l/4,l'-l/4)$ with $|J_l|=l'-l/2>l/2>\bar L$. Hence one gets~\eqref{eq:gstr27}.


The proof of~\eqref{eq:gstr28} proceeds using the $L$-periodicity of the contributions.

\end{proof}

\section{Proof of Theorem~\ref{thm:main}}
\subsection{Setting the parameters}\label{Ss:param}
The sets defined in the proof and the main estimates will depend on a set of parameters $l,\delta,\rho,M, \eta$ and $\tau$.  Our aim now is to fix such parameters, making explicit their dependence on each other. We will refer to such choices during the proof of the main theorem.

\begin{enumerate}
	\item We first fix $\eta_0,\tau_0$ as in Lemma~\ref{rmk:stimax1}. 
	
	\item Then we choose $0<\eps\ll1$ such that 
	\begin{equation}
		\label{eq:eps}
		\eps<-e_\tau(h^*_\tau)/4,\qquad2\eps C<\frac12\min\Bigl\{-\frac{e_\tau(h^*_\tau)}{2},1\Bigr\},
	\end{equation}
where $C$ is as in Theorem \ref{T:1d} and we let $\bar L>0$, $\bar \tau_4>0$ as in Theorem \ref{cor:conv} such that moreover $C/\bar L<1$. In particular, $\eps C/\bar L<-e_\tau(h^*_\tau)/4$.
	
	
	\item  Let then  $l>0$ s.t.
	\begin{equation}\label{eq:lfix}
		l\geq \max\Big\{2\bar L,\frac{\bar C}{-e_\tau(h^*_\tau)/4}\Big\}\geq\max\Big\{2\bar L,\frac{\bar C}{-e_\tau(h^*_\tau)/2-\eps C/\bar L}\Big\},
	\end{equation}
	where  $\bar L$, $C$  are the constant appearing in Theorem \ref{cor:conv}, $\eps$ is as in \eqref{eq:eps} and $\bar C$ is the constant appearing in \eqref{eq:eqtoBeshown_integral}.

	\item We  find  the parameters  ${\varepsilon}_2 = {\varepsilon}_2(\eta_0,\tau_0)$ and ${\tau}_2 = {\tau}_2(\eta_0, \tau_0)$ as in Proposition~\ref{lemma:stimaContributoVariazionePiccola}.
	
	\item We consider then  $\bar \varepsilon \leq {\varepsilon}_2$, $\tau \leq\min\{{\tau}_2,\bar \tau_4\}$ as in Lemma~\ref{lemma:stimaLinea}. We define   $\delta$ as $\delta  = \frac{\bar \varepsilon^d}{16}$. Moreover,  by choosing $\bar \varepsilon$ sufficiently small we can additionally assume that
	\begin{equation}\label{eq:deltafix2}
		D^i_{\eta}(E,Q_l(z))\leq\delta\text{ and }D^j_\eta(E,Q_l(z))\leq\delta,\:i\neq j\quad\Rightarrow\quad\min\{|E\cap Q_l(z)|,|E^c\cap Q_l(z)|\}\leq l^{d-1}.  
	\end{equation}
	The above  follows from Remark~\ref{rmk:lip} (ii). 
	
	\item By Remark~\ref{rmk:lip} (i), we then fix
	\begin{equation}\label{eq:rhofix}
		\rho\sim\delta l. 
	\end{equation}
	in such a way that  for any $\eta$ the following holds
	\begin{equation}\label{eq:rhofix2}
		\forall\,z,z'\text{ s.t. }D_{\eta}(E,Q_l(z))\geq\delta,\:|z-z'|_\infty\leq\rho\quad\Rightarrow\quad D_\eta(E,Q_l(z'))\geq\delta/2.
	\end{equation}

	\item Then we fix $M$ such that
	\begin{equation}
		\label{eq:Mfix}
		\frac{M\rho}{2d}>C_1l +1,
	\end{equation}
	where $C_1$ is the constant appearing in Lemma~\ref{lemma:stimaLinea}.
	
	\item By applying Proposition~\ref{lemma:local_rigidity_alpha}, we obtain  $\bar\eta=\bar{\eta}(M,l)$  and ${\tau}_1 = {\tau}_1(M,l,\delta/2)$.  Thus we fix
	\begin{equation}\label{eq:etafix}
		0<\eta<\bar{\eta},\quad \bar\eta=\bar{\eta}(M,l).
	\end{equation}

	\item Finally, we choose  $\bar\tau>0$ s.t.
	\begin{equation}
		\label{eq:taufix0}
			\bar\tau<\tau_0,  \qquad\text{$\tau_0$ as in Lemma~\ref{rmk:stimax1},}
	\end{equation}
\begin{equation}
	\bar{\tau}<\bar{\tau}_4,\qquad\text{$\bar \tau_4$ as in Theorem~\ref{cor:conv}},
\end{equation}
	\begin{equation}
		\label{eq:taufix1}
		\bar\tau<{\tau}_2, \,{\tau}_2\text{ as in Proposition~\ref{lemma:stimaContributoVariazionePiccola} and Lemma~\ref{lemma:stimaLinea}},
	\end{equation}
	\begin{equation}
		\label{eq:taufix2}
		\bar\tau<{\tau}_1, \text{ ${\tau}_1$ as in Proposition~\ref{lemma:local_rigidity_alpha} depending on $M,l,\delta/2$}.
	\end{equation}
	
\end{enumerate}

 By $[0,L)^d$-periodicity  of $E$ we will denote by $[0,L)^d$  the cube of size $L$ with the usual identification of the boundary.

\subsection{ Decomposition of $[0,L)^d$} \label{subsec:dec}

Now we perform a decomposition of $[0,L)^d$ into different sets according to the $L^1$ closeness of the minimizer $E$ to stripes orthogonal to the different coordinate axes. The construction of this decomposition, in comparison to the one introduced in~\cite{dr_arma,dr_siam,ker}, has to take into account the boundary effects on the slices in direction $e_i$ when close to stripes with boundaries orthogonal to $e_i$.

Let us now consider any $L>l\geq 2\bar L$ as in Point 1. of Section \ref{Ss:param}. We will have that
$[0,L)^d =A_{-1}\cup A_0 \cup (B\setminus B_l)\cup A_{1,l}\cup\ldots \cup A_{d,l}$ where
\begin{itemize}
	\item $A_{i,l}$ with $i > 0$ is made of points $z$ such that there is only one direction $e_i$ such that $E_\tau\cap Q_{l}(z)$ is close to stripes with boundaries orthogonal to~$e_i$. 
	\item $A_{-1}$ is a set of points $z$ such that $E_\tau\cap Q_{l}(z)$ is close both to stripes with boundaries orthogonal to $e_i$ and to stripes with boundaries orthogonal to $e_j$ for some $i\neq j$.  In particular, by Remark~\ref{rmk:lip} (ii) one has that either $|E_\tau\cap Q_l(z)|\ll l^d$ or $|E_{\tau}^c\cap Q_l(z)|\ll l^d$.
	\item $B\setminus B_l$ is a suitable set of points close to the boundaries of the sets $A_{i,l}$ as $i\in\{1,\dots,d\}$. 
	\item $A_{0}$ is a set of points $z$ where none of the above points is true, namely the set $E$ is far from stripes in any direction.
\end{itemize}

The aim is to show that $A_0\cup A_{-1}\cup B\setminus B_l = \emptyset$ and  that there exists only one $A_{i,l}$ with $i >  0$.

Let us first define the sets $A_i$, for $i\in\{-1,0,1,\ldots,d\}$.

We preliminarily define
\begin{equation*}
	\begin{split}
		\tilde{A}_{0}:= \insieme{ z\in [0,L)^d:\ D_{\eta}(E,Q_{l}(z)) \geq \delta }.
	\end{split}
\end{equation*}
Hence, by the choice of $\delta,M$ made in Section~\ref{Ss:param} and by Proposition~\ref{lemma:local_rigidity_alpha}, for every $z\in \tilde{A}_{0}$ one has that $\bar{F}_{\tau}(E,Q_{l}(z)) > M$.

Let us denote by $\tilde{A}_{-1}$ the set
\begin{equation*}
	\begin{split}
		\tilde{A}_{-1}: = \insieme{z\in [0,L)^d: \exists\, i,j \text{ with } i\neq j \text{ \st }\, D^{i}_{\eta} (E,Q_{l}(z))\leq\delta , D^{j}_{\eta} (E,Q_{l}(z)) \leq \delta }.
	\end{split}
\end{equation*}

Since $\delta$ satisfies~\eqref{eq:deltafix2}, when $z\in \tilde{A}_{-1}$, then one has that $\min (|E\cap Q_{l}(z)|, |Q_{l}(z)\setminus E|) \leq  l^{d-1} $.
Thus, using  Lemma~\ref{lemma:stimaQuasiPieno} with $\delta=1/l$, one has that
\begin{equation*}
	\begin{split}
		\bar{F}_{\tau}(E, Q_{l}(z)) \geq  -\frac{d}{l\eta_0}.
	\end{split}
\end{equation*}

The sets $\tilde A_0$ and $\tilde A_{-1}$ can be enlarged while keeping analogous properties.
Indeed, by the choice of $\rho$ made in~\eqref{eq:rhofix},~\eqref{eq:rhofix2} holds, namely for every $z\in \tilde{A}_{0}$ and $|z- z' |_\infty\leq\rho$ one has that $D_{\eta}(E,Q_{l}(z')) > \delta/2$.
Moreover, let now $z'$ such that $|z- z' |_\infty\leq 1$ with $z\in \tilde{A}_{-1}$. It is not difficult to see that if $|Q_{l}(z)\setminus E | \leq l^{d-1}$ then $|Q_{l}(z')\setminus E| \lesssim l^{d-1}$. Thus from Lemma~\ref{lemma:stimaQuasiPieno}, one has that
\begin{equation}
	\label{eq:tildeC}
	\begin{split}
		\bar{F}_{\tau}(E, Q_{l}(z')) \geq -\frac{\tilde C_d}{l\eta_0}.
	\end{split}
\end{equation}

The above observations motivate the following definitions
\begin{align}
	A_{0} &:= \insieme{ z' \in [0,L)^d: \exists\, z \in \tilde{A}_{0}\text{ with }|z-z'|_{\infty}  \leq \rho }\label{a0}\\
	A_{-1} &:= \insieme{ z' \in [0,L)^d: \exists\, z \in \tilde{A}_{-1}\text{ with }|z-z' |_{\infty}  \leq 1 },\label{a1}
\end{align}

By the choice of the parameters and the observations above, for every $z\in A_{0}$ one has that $\bar{F}_{\tau}(E,Q_{l}(z)) > M$ and for every $z\in A_{-1}$, $\bar{F}_{\tau}(E,Q_{l}(z)) \geq-\tilde C_d/(l\eta_0)$.

Let us denote by $A:= A_{0}\cup A_{-1}$. 

The set $[0,L)^d\setminus A$ has the following property: for every $z\in [0,L)^d\setminus A$, there exists $i\in \{ 1,\ldots,d\}$ such that $D^{i}_{\eta}(E,Q_{l}(z)) \leq \delta$ and for every $k\neq i$ one has that $D^{k}_{\eta}(E,Q_{l}(z)) > \delta$.

Given  that $A$ is closed, we consider the connected components $\mathcal C_{1},\ldots,\mathcal C_{n}$ of $[0,L)^d\setminus A$.  The sets $\mathcal C_{i}$ are path-wise connected. 
Moreover, given a connected component $\mathcal C_{j}$ one has that there exists  $i$ such that $D^{i}_{\eta}(E,Q_{l}(z)) \leq \delta$ for every $z\in\mathcal  C_{j}$  and for every $k\neq i$ one has that $D^{k}_{\eta}(E,Q_{l}(z)) > \delta$.  
We will say that $\mathcal C_j$ is oriented in direction $e_i$ if there is a point in $z\in \mathcal C_j$ such that $D^{i}_\eta(E,Q_{l}(z)) \leq \delta$. 
Because of the above being oriented along direction $e_{i}$ is well-defined.

We will denote by $A_{i}$ the union of the connected  components $\mathcal C_{j}$ such that $\mathcal C_{j}$ is oriented along the direction $e_{i}$. 

We observe the following

\begin{enumerate}[(a)]
	\item The sets $A=A_{-1}\cup A_{0}$, $A_{1}$, $A_{2}$, $\ldots, A_d$  form a partition of $[0,L)^d$. 
	\item The sets $A_{-1}, A_{0}$ are closed and $A_{i}$, $i>0$, are open.  
	\item For every $z\in A_{i}$, we have that $D^{i}_{\eta}(E,Q_{l}(z)) \leq \delta$. 
	\item  There exists $\rho$ (independent of $L,\tau$) such that  if $z\in A_{0}$, then $\exists\,z'$ s.t. $Q_{\rho}(z')\subset A_{0}$ and $z \in Q_{\rho}(z')$. If $z\in A_{-1}$ then $\exists\,z'$ s.t. $Q_{1}(z')\subset A_{-1}$ and $z \in Q_{1}(z')$. 
	\item For every $z\in {A}_{i}$ and $z'\in {A}_{j}$ one has that there exists a point $\tilde{z}$ in the segment connecting $z$ to $z'$ lying in ${A}_{0}\cup A_{-1}$. 
\end{enumerate}

Let now $B = \bigcup_{i> 0}A_{i}$, $A=A_0\cup A_{-1}$. 

From conditions $\mathrm{(b)}$ and $\mathrm{(e)}$ above, $B_{t^{\perp}_{i}}$ is a finite union of intervals, each belonging to some $A_{i,t_i^\perp}$, $i\in\{1,\dots,d\}$. Moreover, by $\mathrm{(d)}$, for every point that does not belong to $B_{t^\perp_{i}}$ there is a neighbourhood of fixed positive size that is not included in $B_{t^\perp_i}$. 
Let $\{ I^j_{1},\ldots,I^j_{n(j,t_i^\perp)}\}$ such that $\bigcup_{\ell=1}^{n(j,t_i^\perp)} I^j_{\ell} = A_{j,t_i^\perp}$ with $I^j_\ell \cap I^i_{k} = \emptyset$ whenever $j\neq i$ or $j=i$ and $\ell\neq k$. 
We can further assume that $I^j_{\ell} \leq I^j_{\ell+1}$, namely that for every $s\in I^j_{\ell}$ and $s'\in I^j_{\ell+1}$ it holds $s \leq s'$. 
By construction there exists $J_{k} \subset A_{t^{\perp}_{i}}$ such that $I^j_{\ell}\leq  J_{k} \leq I^j_{\ell+1}$, for every $\ell, j$. We set $\bar n(t_i^\perp)=\sum_{j=1}^dn(j,t_i^\perp)$ to be the number of such disjoint intervals $J_k\subset A_{t_i^\perp}$. Whenever $J_k \cap A_{0,t_i^\perp}\neq\emptyset$,  we have that $|J_k | > \rho$  and whenever $J_{k} \cap A_{-1,t^\perp_i}\neq \emptyset $ then $|J_{k}| > 1$. 

Given $i\in\{1,\dots,d\}$, $\ell\in\{1,\dots,n(i,t_i^\perp)\}$ and $I^i_\ell=(a^i_\ell,b^i_\ell)$, define $I^i_{\ell,l}=(a^i_\ell+l/4,b^i_\ell-l/4)$ whenever $|b^i_\ell-a^i_\ell|> l/2$ and $I^i_{\ell,l}=\emptyset$ otherwise. Set also $n(i,t_i^\perp,l)=n(i,t_i^\perp)-\#\{\ell:\,|I^i_{\ell,l}|<l/2\}$. In particular, for all $\ell\in n(i,t_i^\perp,l)$ one has that $|I^i_\ell|\geq l$. Then define 
\begin{equation}\label{eq:ail}
	A_{i,t_i^\perp,l}=\underset{\ell=1}{\overset{n(i,t_i^\perp,l)}{\bigcup}} I^i_{\ell,l}
\end{equation}
and 
\begin{equation}\label{eq:albl}
	A_{i,l}=\underset{\{t_{i}^\perp\in[0,L)^{d-1}\}}{\bigcup}A_{i,t_i^\perp,l},\qquad B_{l}=\underset{i=1}{\overset{d}{\bigcup}}A_{i,l}.
\end{equation}

Thus we get the partition $[0,L)^d=A_0\cup A_{-1}\cup (B\setminus B_l)\cup A_{1,l}\cup\ldots\cup A_{d,l}$.

\subsection{Proof of Theorem~\ref{thm:main}}\label{Ss:main}

\textbf{Step 1} First we show the following estimate 

\begin{align}
		\frac{1}{L^d} \int_{B_{t^{\perp}_{i}}} \bar{F}_{i,\tau}(E,&Q_{l}(t^{\perp}_{i}+se_i))\ds + \frac1{dL^d} \int_{A_{t^{\perp}_{i}}}\bar{F}_{\tau}(E,Q_{l}(t^{\perp}_{i}+se_i)) \ds  \notag\\
		&\geq \sum_{\ell\in n(i,t_i^\perp,l)}\frac{e_\tau(h_\tau(I^i_{\ell,l}))|I^i_{\ell,l}|}{L^d} - \frac{C_0n(i,t_i^\perp,l)}{L^d}- C(d,\eta_0) \frac{|A_{t^\perp_i}|}{l L^d}+\frac{\#\{A_{0,t_i^\perp}\cap\partial A_{i,t_i^\perp}\}}{L^d}.	\label{eq:toBeShown_slice}
\end{align}

By the  definitions given in Section~\ref{subsec:dec}, one has that
\begin{equation*}
	\begin{split}
		\frac{1}{L^d} \int_{B_{t^{\perp}_{i}}} \bar{F}_{i,\tau}(E,&Q_{l}(t^{\perp}_{i}+se_i)) \ds  + 
		\frac{1}{d L^d} \int_{A_{t^{\perp}_{i}}} \bar{F}_{\tau}(E,Q_{l}(t^{\perp}_{i}+se_i)) \ds 
		\\ & \geq \sum_{j=1}^d\sum_{\ell=1}^{n(j,t_i^\perp)}  \frac{1}{L^d}\int_{I^j_{\ell}} \bar{F}_{i,\tau}(E,Q_{l}(t^{\perp}_{i}+se_i)) \ds
		+ \frac{1}{dL^d}\sum_{\ell=1}^{\bar n(t_i^\perp)} \int_{J_{\ell}} \bar{F}_{\tau}(E,Q_{l}(t^{\perp}_i+se_i)) \ds 
		\\ & \geq \frac{1}{L^d}\sum_{j=1}^d\sum_{\ell=1}^{n(j,t_i^\perp)} \Big( \int_{I^j_{\ell}} \bar{F}_{i,\tau}(E,Q_{l}(t^{\perp}_{i}+se_i)) \ds
		+ \frac{1}{2d} \int_{J_{k(j,\ell)-1}\cup J_{k(j,\ell)}} \bar{F}_{\tau}(E,Q_{l}(t^{\perp}_i+se_i)) \ds\Big),
	\end{split}
\end{equation*}
where in the second inequality we have used the $[0,L)^d$-periodicity and the convention $J_1:=J_{\bar n(t_i^\perp)}$.

Let us first consider $I^i_{\ell} \subset A_{i,t_i^\perp}$.  
By construction, we have that $\partial I^i_{\ell}\subset A_{t^\perp_i}$. 

If $\partial I^i_{\ell}\subset A_{-1,t^\perp_i}$, by using our choice of parameters we can apply~\eqref{eq:gstr27} in Lemma~\ref{lemma:stimaLinea} and obtain
\begin{equation*}
	\begin{split}
		\frac{1}{L^d}\int_{I^i_{\ell}} \bar{F}_{i,\tau}(E,Q_{l}(t^{\perp}_{i}+se_i))\ds  \geq \frac{1}{L^d}\Big[\Big(e_\tau(h_\tau(I^i_{\ell,l}))|I^i_{\ell,l}| - C_0\Big)\chi_{(0, +\infty)}(|I^i_\ell|-l) -\frac{C_1} l\Big].
	\end{split}
\end{equation*}

If $\partial I^i_\ell \cap A_{0, t^\perp_i}\neq \emptyset$, by using our choice of parameters, namely~\eqref{eq:lfix} and~\eqref{eq:taufix1}, we can apply~\eqref{eq:gstr36} in Lemma~\ref{lemma:stimaLinea}, and obtain
\begin{equation*}
	\begin{split}
		\frac{1}{L^d}\int_{I^i_{\ell}} \bar{F}_{i,\tau}(E,Q_{l}(t^{\perp}_{i}+se_i))\ds \geq\frac{1}{L^d}\Big[\Big(e_\tau(h_\tau(I^i_{\ell,l}))|I^i_{\ell,l}| - C_0\Big)\chi_{(0, +\infty)}(|I^i_\ell|-l) -{C_1} l\Big].
	\end{split}
\end{equation*}

On the other hand, if $\partial I^i_\ell \cap A_{0,t^\perp_i}\neq \emptyset$, we have that either $J_{k(i,\ell)}\cap A_{0,t^\perp_i}\neq \emptyset$ or $J_{k(i,\ell)-1}\cap A_{0,t^\perp_i}\neq\emptyset$. Thus
\begin{equation*}
	\begin{split}
		\frac{1}{2dL^d}\int_{J_{k(i,\ell)-1}} \bar{F}_{\tau}(E,Q_{l}(t^{\perp}_{i}+se_i)) \ds & + \frac{1}{2dL^d}\int_{J_{k(i,\ell)}} \bar{F}_{\tau}(E,Q_{l}(t^{\perp}_{i}+se_i)) \ds  \\ &\geq  \frac{M\rho}{2dL^d}  - \frac{|J_{k(i,\ell)-1}\cap A_{-1,t^\perp_i} |\tilde C_d}{2dl\eta_0 L^d} - \frac{|J_{k(i,\ell)}\cap A_{-1,t^\perp_i} |\tilde C_d}{2dl \eta_0L^d},
	\end{split}
\end{equation*}
where $\tilde C_d$ is the  constant in~\eqref{eq:tildeC}.

Since $M$ satisfies~\eqref{eq:Mfix}, in both cases $\partial I^i_{\ell}\subset A_{-1,t^\perp_i}$ or $\partial I^i_{\ell}\cap A_{0,t^\perp_i}\neq \emptyset$, we have that
\begin{align}
		\frac{1}{L^d}\int_{I^i_{\ell}} &\bar{F}_{i,\tau}(E,Q_{l}(t^\perp_i+se_i)) \ds + 
		\frac{1}{2dL^d}\int_{J_{k(i,\ell)-1}}  \bar{F}_{\tau}(E,Q_{l}(t^{\perp}_{i}+se_i))\ds
		+ \frac{1}{2dL^d}\int_{J_{k(i,\ell)}}  \bar{F}_{\tau}(E,Q_{l}(t^{\perp}_{i}+se_i))\ds\notag\\
		&\geq \Bigl(\frac{e_\tau(h_\tau(I^i_{\ell,l}))|I^i_{\ell,l}|}{L^d} -\frac{C_0}{L^d}\Bigr)\chi_{(0,+\infty)}(|I^i_{\ell}|-l)\notag\\
		&+\frac{\#\{A_{0,t_i^\perp}\cap\partial I^i_\ell\}}{L^d}- \frac{|J_{k(i,\ell)-1}\cap A_{-1,t^\perp_i}|\tilde C_d}{2dl\eta_0L^d}
		- \frac{|J_{k(i,\ell)}\cap A_{-1,t^\perp_i}|\tilde C_d}{2dl\eta_0L^d}.\label{eq:a0}
\end{align}

If $I^j_{\ell} \subset A_{j,t^\perp_i}$ with $j \neq i$ from  Lemma~\ref{lemma:stimaLinea} Point (i) it holds
\begin{equation*}
	\begin{split}
		\frac 1{L^d}\int_{I^j_{\ell}} \bar{F}_{i,\tau}(E,Q_{l}(t^{\perp}_{i}+se_i)) \ds\geq  - \frac{C_1}{lL^d}.
	\end{split}
\end{equation*}

In general for every $J_{k}$  we have that 
\begin{equation*}
	\begin{split}
		\frac{1}{dL^d}\int_{J_{k}} \bar{F}_{\tau}(E,Q_{l}(t^{\perp}_{i}+se_i))\, \ds \geq   \frac{|J_{k}\cap A_{0,t^\perp_i} | M}{dL^d} - \frac{\tilde C_d}{dl\eta_0L^d }|J_{k}\cap A_{-1,t^\perp_i}|. 
	\end{split}
\end{equation*}

For $I^j_{\ell}\subset A_{j,t^\perp_i}$ such that $(J_{k(j,\ell)} \cup J_{k(j,\ell)-1})\cap A_{0,t^\perp_i}\neq \emptyset$ with $j\neq i$, we have that 
\begin{equation*}
	\begin{split}
		\frac{1}{L^d}\int_{I^j_{\ell}} \bar{F}_{i,\tau}(E,Q_{l}(t^\perp_i+se_i)) \ds &+ 
		\frac{1}{2dL^d}\int_{J_{k(j,\ell)-1}}  \bar{F}_{\tau}(E,Q_{l}(t^{\perp}_{i}+se_i))\ds
		+ \frac{1}{2dL^d}\int_{J_{k(j,\ell)}}  \bar{F}_{\tau}(E,Q_{l}(t^{\perp}_{i}+se_i))\ds\\
		&\geq -\frac{C_1}{lL^d} + \frac{M\rho}{2dL^d} - \frac{|J_{k(j,\ell)-1}\cap A_{-1,t^\perp_i}|\tilde C_d}{2dl\eta_0L^d}
		- \frac{|J_{k(j,\ell)}\cap A_{-1,t^\perp_i}|\tilde C_d}{2dl\eta_0L^d}.
		\\ &\geq
		\frac{\#\{A_{0,t_i^\perp}\cap \partial I^j_\ell\}}{L^d}- \frac{|J_{k(j,\ell)-1}\cap A_{-1,t^\perp_i}|\tilde C_d}{2dl\eta_0L^d}
		- \frac{|J_{k(j,\ell)}\cap A_{-1,t^\perp_i}|\tilde C_d}{2dl\eta_0L^d}.
	\end{split}
\end{equation*}
where the last inequality is true due to~\eqref{eq:Mfix}.

For $I^j_{\ell}\subset A_{j,t^\perp_i}$ such that $(J_{k(j,\ell)} \cup J_{k(j,\ell)-1})\subset  A_{-1,t^\perp_i}$ with $j\neq i$, we have that 
\begin{align}\label{eq:a-1}
		\frac{1}{L^d}\int_{I^j_{\ell}} \bar{F}_{i,\tau}(E,Q_{l}(t^\perp_i+se_i))& \ds + 
		\frac{1}{2dL^d}\int_{J_{k(j,\ell)-1}}  \bar{F}_{\tau}(E,Q_{l}(t^{\perp}_{i}+se_i))\ds
		+ \frac{1}{2dL^d}\int_{J_{k(j,\ell)}}  \bar{F}_{\tau}(E,Q_{l}(t^{\perp}_{i}+se_i))\ds\notag\\
		&\geq -\frac{C_1}{lL^d}   - \frac{|J_{k(j,\ell)-1}\cap A_{-1,t^\perp_i}|\tilde C_d}{2dl\eta_0L^d}
		- \frac{|J_{k(j,\ell)}\cap A_{-1,t^\perp_i}|\tilde C_d}{2dl\eta_0L^d}.
		\notag\\ &\geq
		- \max\Big(C_1,\frac{\tilde C_d}{\eta_0d}\Big)\bigg(\frac{|J_{k(j,\ell)-1}\cap A_{-1,t^\perp_i}|}{lL^d}
		+ \frac{|J_{k(j,\ell)}\cap A_{-1,t^\perp_i}|}{lL^d}\bigg).
\end{align}
where in the last inequality we have used that $|J_{k(j,\ell)}\cap A_{-1,t^\perp_i}|\geq1, \,|J_{k(j,\ell)-1}\cap A_{-1,t^\perp_i}|\geq1$.

Summing \eqref{eq:a0} and \eqref{eq:a-1} over $j\in\{1,\dots,d\}$, and taking
\begin{equation}\label{eq:ca0}
	C(d,\eta_0)=\max\Big(C_1,\frac{\tilde C_d}{\eta_0d}\Big), 
\end{equation} one obtains~\eqref{eq:toBeShown_slice} as desired.

\textbf{Step 2}

Our aim is to deduce from~\eqref{eq:toBeShown_slice} the following lower bound

\begin{equation}\label{eq:eqtoBeshown_integral}
	\Fcal_{\tau,L}(E)\geq\sum_{i=1}^d\int_{[0,L)^{d-1}}\sum_{\ell\in n(i,t_i^\perp,l)}e_\tau(h_\tau(I^i_{\ell,l}))\frac{|I^i_{\ell,l}|}{L^d}\dt_i^\perp-\frac{\bar C}{lL^d}|B_l^c|+\sum_{i=1}^d\int_{[0,L)^{d-1}}\frac{\#\{A_{0,t_i^\perp}\cap\partial A_{i,t_i^\perp}\}}{L^d}\dt_i^\perp,
\end{equation}
where $\bar C=2C_0+dC(d,\eta_0)$.

Integrating~\eqref{eq:toBeShown_slice} w.r.t. $t_i^\perp\in[0,L)^{d-1}$  one has that

\begin{align}
	\frac{1}{L^d}\int_{A_i}\bar F_{i,\tau}(E,Q_l(z))\dz&+\frac{1}{dL^d}\int_A\bar F_{\tau}(E,Q_l(z))\dz\geq\int_{[0,L)^{d-1}}\sum_{\ell\in n(i,t_i^\perp,l)}e_\tau(h_\tau(I^i_{\ell,l}))\frac{|I^i_{\ell,l}|}{L^d}\dt_i^\perp\notag\\
	&-\frac{C_0}{L^d}\int_{[0,L)^{d-1}}n(i,t_i^\perp,l)\dt_i^\perp-\frac{C(d,\eta_0)}{lL^d}|A|\notag\\
	&+\int_{[0,L)^{d-1}}\frac{\#\{A_{0,t_i^\perp}\cap\partial A_{i,t_i^\perp}\}}{L^d}\dt_i^\perp.
\end{align}
Summing the above over $i\in\{1,\dots,d\}$ and using the lower bound~\eqref{eq:gstr14} together with the definition of the sets in the decomposition  one obtains
\begin{align}
\Fcal_{\tau,L}(E)&\geq\sum_{i=1}^d\int_{[0,L)^{d-1}}\sum_{\ell\in n(i,t_i^\perp,l)}e_\tau(h_\tau(I^i_{\ell,l}))\frac{|I^i_{\ell,l}|}{L^d}\dt_i^\perp-\frac{C_0}{L^d}\sum_{i=1}^d\int_{[0,L)^{d-1}}n(i,t_i^\perp,l)\dt_i^\perp
-\frac{dC(d,\eta_0)}{lL^d}|A|\notag\\
&+\sum_{i=1}^d\int_{[0,L)^{d-1}}\frac{\#\{A_{0,t_i^\perp}\cap\partial A_{i,t_i^\perp}\}}{L^d}\dt_i^\perp\notag\\
&\geq\sum_{i=1}^d\int_{[0,L)^{d-1}}\sum_{\ell\in n(i,t_i^\perp,l)}e_\tau(h_\tau(I^i_{\ell,l}))\frac{|I^i_{\ell,l}|}{L^d}\dt_i^\perp-\frac{\bar C}{lL^d}|B_l^c|+\sum_{i=1}^d\int_{[0,L)^{d-1}}\frac{\#\{A_{0,t_i^\perp}\cap\partial A_{i,t_i^\perp}\}}{L^d}\dt_i^\perp,\notag\\
\end{align}
where in the last inequality we observed that $|B\setminus B_l|\geq \frac{l}{2}\sum_{i=1}^d\int_{[0,L)^{d-1}}n(i,t_i^\perp,l)\dt_i^\perp$.

\textbf{Step 3}

Now let us assume that $E$ is a minimizer, namely $\Fcal_{\tau,L}(E)=e_\tau(h_{\tau,L})$, and that  ${|B_l^c|}\neq0$. 

First of all, we claim that the following estimate holds: for all $i=1,\ldots,d$ and for all $t_i^\perp\in[0,L)^{d-1}$ it holds 
\begin{align}
	\sum_{k=1}^{ n(i,t_i^\perp,l)}e_\tau(h_\tau(I^i_{\ell_k,l}))\frac{|I^i_{\ell_k,l}|}{|A_{i,t_i^\perp,l}|}\geq e_\tau(h_\tau([0,|A_{i,t_i^\perp,l}|]))-\frac{\eps C}{|A_{i,t_i^\perp,l}|},\label{eq:claim1}
\end{align}
where $\eps$ satisfies the conditions in \eqref{eq:eps} and $C$ is the constant appearing in Theorem \ref{cor:conv}. Indeed, by minimality of $e_\tau(h^*_\tau)=e_{\infty,\tau}$, the fact that $\sum_{k=1}^{ n(i,t_i^\perp,l)}|I^i_{\ell_k,l}|=|A_{i,t_i^\perp,l}|$ and by \eqref{eq:epsineq}, one has that
\begin{align*}
		\sum_{k=1}^{ n(i,t_i^\perp,l)}e_\tau(h_\tau(I^i_{\ell_k,l}))\frac{|I^i_{\ell_k,l}|}{|A_{i,t_i^\perp,l}|}\geq e_\tau(h^*_{\tau})\geq e_\tau(h_\tau([0,|A_{i,t_i^\perp,l}|]))-\frac{\eps C}{|A_{i,t_i^\perp,l}|}.
\end{align*}
Using the fact that $\Fcal_{\tau, L}L^d=e(h_{\tau,L}) L^d=e(h_{\tau,L})|B_l|+e(h_{\tau,L})|B_l^c|$, inequality \eqref{eq:eqtoBeshown_integral} rewrites as
\begin{align}
	\sum_{i=1}^d\int_{[0,L)^{d-1}}\Big[e_\tau(h_{\tau,L})|A_{i,t_i^\perp,l}|-\sum_{\ell\in n(i,t_i^\perp,l)}e_\tau(h_\tau(I^i_{\ell,l}))|I^i_{\ell,l}|\Big]\dt_i^\perp&\geq\sum_{i =1}^d\int_{[0,L)^{d-1}}\Big(-e_\tau(h_{\tau,L})-\frac{\bar C}{l}\Big)|A^c_{i,t_i^\perp,l}|\dt_i^\perp\notag\\
&+\sum_{i=1}^d\int_{[0,L)^{d-1}}\#\{A_{0,t_i^\perp}\cap\partial A_{i,t_i^\perp}\}\dt_i^\perp.	
\end{align}
Using in the above the lower bound \eqref{eq:claim1} one has that
\begin{align}
	\sum_{i=1}^d\int_{[0,L)^{d-1}}\Big[e_\tau(h_{\tau,L})-e_\tau(h_\tau([0,|A_{i,t_i^\perp,l}|]))\Big]|A_{i,t_i^\perp,l}|\dt_i^\perp&\geq-\eps C\sum_{i=1}^d\bigl|\bigl\{t_i^\perp\in[0,L)^{d-1}:\,A_{i,t_i^\perp,l}\neq[0,L)\bigr\}\bigr|\notag\\
	&+\sum_{i =1}^d\int_{[0,L)^{d-1}}\Big(-e_\tau(h_{\tau,L})-\frac{\bar C}{l}\Big)|A^c_{i,t_i^\perp,l}|\dt_i^\perp\notag\\
	&+\sum_{i=1}^d\int_{[0,L)^{d-1}}\#\{A_{0,t_i^\perp}\cap\partial A_{i,t_i^\perp}\}\dt_i^\perp.\label{eq:ineqfinal}
\end{align}
Since both $L$ and $|A_{i,t_i^\perp,l}|$ are greater than $\bar L$, by Theorem \ref{cor:conv} one has that 
\begin{equation}\label{eq:taua}
\Big|e_\tau(h_{\tau,L})-e_\tau(h_\tau([0,|A_{i,t_i^\perp,l}|]))\Big|\leq \frac{\eps C}{|A_{i,t_i^\perp,l}|}.
\end{equation}
Hence, \eqref{eq:ineqfinal} and \eqref{eq:taua} imply that 
\begin{align}
	\eps C\sum_{i=1}^d\bigl|\bigl\{t_i^\perp\in[0,L)^{d-1}:\,A_{i,t_i^\perp,l}\neq[0,L)\bigr\}\bigr|&\geq-\eps C\sum_{i=1}^d\bigl|\bigl\{t_i^\perp\in[0,L)^{d-1}:\,A_{i,t_i^\perp,l}\neq[0,L)\bigr\}\bigr|\notag\\
	&+\sum_{i =1}^d\int_{[0,L)^{d-1}}\Big(-e_\tau(h_{\tau,L})-\frac{\bar C}{l}\Big)|A^c_{i,t_i^\perp,l}|\dt_i^\perp\notag\\
		&+\sum_{i=1}^d\int_{[0,L)^{d-1}}\#\{A_{0,t_i^\perp}\cap\partial A_{i,t_i^\perp}\}\dt_i^\perp.\label{eq:5.25}
\end{align}
By Theorem \ref{cor:conv}, the assumptions \eqref{eq:eps} on $\eps$ and \eqref{eq:lfix} on $l$ one has that
\begin{equation}
-e_\tau(h_{\tau,L})-\frac{\bar C}{l}\geq-e_\tau(h^*_\tau)-\frac{\eps C}{\bar L}-\frac{\bar C}{l}\geq -\frac34 e_{\tau}(h^*_\tau)-\frac{\bar C}{l}\geq-\frac{ e_{\tau}(h^*_\tau)}{2}>0.	\label{eq:5.26}
\end{equation}
Hence from \eqref{eq:5.25} and \eqref{eq:5.26} one obtains
\begin{align}
	2\eps C\sum_{i =1}^d\bigl|\bigl\{t_i^\perp\in[0,L)^{d-1}:\,A_{i,t_i^\perp,l}\neq[0,L)\bigr\}\bigr|&\geq -\frac{e_\tau(h^*_\tau)}{2}\sum_{i =1}^d\int_{[0,L)^{d-1}}|A^c_{i,t_i^\perp,l}|\dt_i^\perp\notag\\
	&+\sum_{i=1}^d\int_{[0,L)^{d-1}}\#\{A_{0,t_i^\perp}\cap\partial A_{i,t_i^\perp}\}\dt_i^\perp.
\end{align}
Now notice that whenever for some $t_i^\perp$ one has that $A_{i,t_i^\perp,l}\neq[0,L)$ or equivalently $A^c_{i,t_i^\perp,l}\neq\emptyset$, then either $A_{0,t_i^\perp}\cap \partial A_{i,t_i^\perp}\neq\emptyset$ and $\#\{A_{0,t_i^\perp}\cap \partial A_{i,t_i^\perp}\}\geq1$ or $A_{-1,t_i^\perp}\neq \emptyset$ and in particular $|A^c_{i,t_i^\perp,l}|\geq|A_{-1,t_i^\perp}|\geq 1$.

 Thus 
\begin{align}
		2\eps C\sum_{i =1}^d\bigl|\bigl\{t_i^\perp\in[0,L)^{d-1}:\,A_{i,t_i^\perp,l}\neq[0,L)\bigr\}\bigr|&\geq\sum_{i =1}^d\Bigl[-\frac{e_\tau(h^*_\tau)}{2}\bigl|\bigl\{t_i^\perp\in[0,L)^{d-1}:\,A_{-1,t_i^\perp,l}\neq\emptyset\bigr\}\bigr|\notag\\
		&+\bigl|\bigl\{t_i^\perp\in[0,L)^{d-1}:\,A_{0,t_i^\perp,l}\neq\emptyset\bigr\}\bigr|\Bigr]\notag\\
		&\geq\min\Bigl\{-\frac{e_\tau(h^*_\tau)}{2},1\Bigr\}\sum_{i =1}^d\bigl|\bigl\{t_i^\perp\in[0,L)^{d-1}:\,A_{i,t_i^\perp,l}\neq[0,L)\bigr\}\bigr|.\label{eq:lastineq}
\end{align}
If $|B_l^c|\neq0$, then $\sum_{i=1}^d\bigl|\bigl\{t_i^\perp\in[0,L)^{d-1}:\,A_{i,t_i^\perp,l}\neq[0,L)\bigr\}\bigr|\neq0$ and if $\eps$ satisfies the conditions in \eqref{eq:eps}, then  \eqref{eq:lastineq} is not satisfied, thus reaching a contradiction. 

 Hence, for any minimizer $E$ it holds $|B_l^c|=0$. In particular, $|A|\leq|B^c_l|=0$ and thus by $\mathrm{(e)}$ there is just one $A_i$, $i>0$ with $|A_i|>0$.


We now claim that the fact that there is just one $A_i$, $i>0$ with $|A_i|>0$ proves the statement of Theorem~\ref{thm:main}.

Indeed, let us consider 
\begin{align}
	\frac{1}{L^d}\int_{[0,L)^d}\bar F_{\tau}(E,Q_l(z))\dz&=\frac{1}{L^d}\int_{[0,L)^d}\bar F_{i,\tau}(E,Q_l(z))\dz\label{eq:fi}\\
	&+\frac{1}{L^d}\sum_{j\neq i}\int_{[0,L)^d}\bar F_{j,\tau}(E,Q_l(z))\dz\label{eq:fj}
\end{align}

We apply now Lemma~\ref{lemma:stimaLinea} with $j =i$ and slice the cube $[0,L)^d$ in direction $e_i$. 
From~\eqref{eq:gstr21}, one has that~\eqref{eq:fj} is nonnegative and strictly positive unless the set $E$ is a union of stripes with boundaries orthogonal to $e_i$.
On the other hand, from~\eqref{eq:gstr28}, one has the \rhs of~\eqref{eq:fi} is minimized by a periodic union of stripes with boundaries orthogonal to $e_i$ and with period  $2h_{\tau,L}$ and density $1/2$. Thus, periodic stripes of period $2h_{\tau,L}$ and density $1/2$ are optimal.

\printbibliography

\end{document}